\documentclass{article}

\usepackage{arxiv}

\usepackage[utf8]{inputenc} 
\usepackage[T1]{fontenc}    
\usepackage{hyperref}       
\usepackage{url}            
\usepackage{booktabs}       
\usepackage{amsfonts}       
\usepackage{nicefrac}       
\usepackage{microtype}      
\usepackage{lipsum}
\usepackage{graphicx}
\graphicspath{ {./images/} }
\usepackage{amsthm}
\usepackage{amsmath}
\usepackage{epsfig}
\usepackage{graphicx}
\usepackage{graphics}
\usepackage{bbm}
\usepackage{enumitem}
\usepackage{subfigure}
\usepackage{placeins}
\usepackage{amssymb,srcltx}
\usepackage{color}

\newcommand{\sumip}{\sum^p_{i=1}}

\newcommand{\ii}{i=1,\ldots,p}

\newtheorem{theorem}{Theorem}
\newtheorem{lemma}{Lemma}
\newtheorem{proposition}{Proposition}
\newtheorem{corollary}{Corollary}
\newtheorem{definition}{Definition}
\newtheorem{remark}{Remark}

\newcommand{\x}{\mathbf{x}}

\newcommand{\RR}{\mathbb{R}}

\newcommand{\bmp}{\mathbf{m}}
\newcommand{\MM}{\mathbf{M}}
\newcommand{\cov}{\mathrm{cov}}

\newcommand{\ep}{\mathbf{e}}
\newcommand{\bd}{\mathbf{d}}

\newcommand{\Y}{\mathbf{Y}}
\newcommand{\bY}{\mathbf{Y}}

\newcommand{\zero}{\mathbf{0}}
\newcommand{\one}{\mathbf{1}}

\newcommand{\w}{\mathbf{w}}
\newcommand{\kp}{\mathbf{k}}
\newcommand{\W}{\mathbf{W}}
\newcommand{\V}{\mathbf{V}}
\newcommand{\U}{\mathbf{U}}

\newcommand{\bc}{\mathbf{c}}
\newcommand{\z}{\mathbf{z}}

\newcommand{\A}{\mathbf{A}}
\newcommand{\ap}{\mathbf{a}}
\newcommand{\bp}{\mathbf{b}}
\newcommand{\X}{\mathbf{X}}

\newcommand{\bI}{\mathbf{I}}

\newcommand{\s}{\mathbf{s}}

\newcommand{\xp}{\mathbf{x}}

\newcommand{\TESN}{\textrm{TESN}}
\newcommand{\TEST}{\textrm{TEST}}

\newcommand{\y}{\mathbf{y}}

\newcommand{\dr}[1]{{\mathrm d}#1}
\newcommand{\balpha}{\boldsymbol{\alpha}}
\newcommand{\bmu}{{ \boldsymbol{\mu}}}

\newcommand{\bSigma}{\boldsymbol{\Sigma}}

\newcommand{\bLambda}{\boldsymbol{\Lambda}}
\newcommand{\bbeta}{\boldsymbol{\beta}}
\newcommand{\btheta}{\boldsymbol{\theta}}
\newcommand{\bTheta}{\boldsymbol{\Theta}}
\newcommand{\bxi}{\boldsymbol{\xi}}

\newcommand{\bDelta}{\boldsymbol{\Delta}}

\newcommand{\bPsi}{\boldsymbol{\Psi}}

\newcommand{\btau}{\boldsymbol{\tau}}

\newcommand{\bOmega}{\boldsymbol{\Omega}}

\newcommand{\bgamma}{\boldsymbol{\gamma}}
\newcommand{\bGamma}{\boldsymbol{\Gamma}}

\newcommand{\neta}{\boldsymbol{\eta}}
\newcommand{\blambda}{\boldsymbol{\lambda}}
\newcommand{\bkappa}{\boldsymbol{\kappa}}
\newcommand{\binfty}{\boldsymbol{\infty}}

\newcommand{\PP}{\mathbb{P}}

\newcommand{\tautil}{\tilde\tau}
\newcommand{\CG}[1]{\textcolor{black}{#1}}

\newcommand{\EE}{\mathbbm{E}}
\newcommand{\eqdef}{\overset{\vartriangle}{=}}
\newcommand{\eqdist}{\overset{d}{=}}
\allowdisplaybreaks

\title{Moments of the doubly truncated selection elliptical distributions with emphasis on the unified multivariate skew-$t$ distribution}

\author{
 Christian E. Galarza \\
  Departamento de Estad\'{\i}stica\\
  Escuela Superior Politecnica del Litoral\\
  Guayaquil, Ecuador \\
  \texttt{chedgala@espol.edu.ec} \\
   \And
 Larissa A. Matos \\
  Departamento de Estat\'{\i}stica\\
  Universidade Estadual de Campinas\\
  Campinas, Brazil\\
  \texttt{larissam@unicamp.br} \\
  \And
 Victor H. Lachos \\
  Department of Statistics\\
  University of Connecticut\\
  Storrs CT  06269, U.S.A. \\
  \texttt{hlachos@uconn.edu} \\
}
\begin{document}

\maketitle

\begin{abstract}
 In this paper, we compute doubly truncated moments for the selection elliptical (SE) class of distributions, which includes some multivariate asymmetric versions of well-known elliptical distributions, such as, the normal, Student's $t$, slash, among others. We address the moments for doubly truncated members of this family, establishing neat formulation for high order moments as well as for its first two moments. We establish sufficient and necessary conditions for the existence of these truncated moments. Further, we propose optimized methods able to deal with extreme setting of the parameters, partitions with almost zero volume or no truncation which are validated with a brief numerical study. Finally, we present some results useful in interval censoring models. All results has been particularized to the unified skew-$t$ (SUT) distribution, a complex multivariate asymmetric heavy-tailed distribution which includes the extended skew-$t$ (EST), extended skew-normal (ESN), skew-$t$ (ST) and skew-normal (SN) distributions as particular and limiting  cases.
\end{abstract}


\section{Introduction}
\label{cap4:sec:intro}

Truncated moments have been a topic of high interest in the statistical literature, whose possible applications are wide, from simple to complex statistical models as survival analysis, censored data models, and in the most varied areas of applications such as agronomy, insurance, finance, biology, among others. These areas have data whose inherent characteristics lead to the use of methods that involve these truncated moments, such as restricted responses to a certain interval, partial information such as censoring (which may be left, right or interval), missing, among others. The need to have more flexible models that incorporate features such as asymmetry and robustness, has led to the exploration of this area in last  years. From the first two one-sided truncated moments for the normal distribution, useful in Tobin's model (\cite{tobin1958estimation}), its evolution led to its extension to the multivariate case (\cite{Tallis1961}), double truncation (\cite{bg2009moments}), heavy tails when considering the Student's $t$ bivariate case in \cite{nadarajah2007truncated}, and finally the first two moments for the multivariate Student's $t$ case in \cite{lin2011some}. Besides the interval-type truncation in cases before, \cite{ARISMENDI201729} considers an interesting non-centered ellipsoid elliptical truncation of the form $\ap \leq (\xp-\bmu_\A)^\top \A(\xp-\bmu_\A)$ on well known distributions as the multivariate normal, Student's  $t$, and generalized hyperbolic distribution. On the other hand, \cite{kan2017moments} recently proposed a recursive approach that allows calculating arbitrary product moments for the normal multivariate case. Based on the latter, \cite{roozegar2020moments} proposes the calculation of doubly truncated moments for the normal mean-variance mixture distributions (\cite{barndorff1982normal}) which includes several well-known complex asymmetric multivariate distributions as the generalized hyperbolic distribution (\cite{breymann2013ghyp}).

Unlike \cite{roozegar2020moments}, in this paper we focus our efforts to the general class of asymmetric distributions called the multivariate elliptical selection family. This large family of distributions includes complex multivariate asymmetric versions of well-known elliptical distributions as the normal, Student's $t$, exponential power, hyperbolic, slash, Pearson type II, contaminated normal, among others. We go further in details for the unified skew-$t$ (SUT) distribution, a complex multivariate asymmetric heavy-tailed distribution which includes the extended skew-$t$ (EST) distribution (\cite{arellano2010multivariate}), the skew-$t$ (ST) distribution (\cite{AzzaliniC2003}) and naturally, as limiting cases, its analogous normal and skew-normal (SN) distributions when $\nu \rightarrow \infty$. 

The rest of the paper is organized as follows. In Section 2 we present some preliminaries results, most of them being definitions of the class of distributions and its special cases of interest along the manuscript. Section 3, the addresses the moments for the doubly truncated selection elliptical distributions. Further, we establish formulas for high order moments as well as its first two moments. We present a methodology to deal with some limiting cases and a discussion when a non-truncated partition exists. In addition, we establish sufficient and necessary conditions for the existence of these truncated moments. Section 4 bases results from Section 3 to the SUT case. In Section 5, a brief numerical study is presented in order to validate the methodology. In Section 6, we present some Lemmas and Corollaries related to conditional expectations which are useful in censored modeling. An application of selection elliptical truncated moments on tail conditional expectation is presented in Section 7. Finally, the paper closes with  some conclusions and direction for future research.

\section{Preliminaries}\label{cap4:preli}

\subsection{Selection distributions}

First, we start our exposition defining a selection distribution as in \cite{arellano2006unified}.

\begin{definition}[{\bf selection distribution}]\label{cap4:def1}
Let $\X_1 \in \mathbb{R}^q$ and $\X_2 \in \mathbb{R}^p$ be two random vectors, and denote by $C$ a measurable subset of $\mathbb{R}^q$. We define a selection distribution as the conditional distribution of $\X_2$ given $\X_1 \in C$, that is, as the distribution of $(\X_2 \mid \X_1 \in C)$. We say that a random vector $\Y\in \mathbb{R}^p$ has a selection distribution if $\Y \overset{d}{=}(\X_2\mid \X_1 \in C)$.
\end{definition}

We use the notation $\Y \sim SLCT_{p,q}$ with parameters depending on the characteristics of $\X_1$, $\X_2$, and $C$. Furthermore, for $\X_2$ having a probability density function (pdf) $f_{\X_2}$ say, then $\Y$ has a pdf $f_{\Y}$ given by
\begin{equation}\label{cap4:sel.pdf}
f_{\Y}(\y) = f_{\X_2}(\y)\frac{\PP(\X_1\in C \mid \X_2 = \y)}{\PP(\X_1 \in C)}.
\end{equation}

Since selection distribution depends on the subset $C \in \RR^q$, particular cases are obtained. One of the most important case is when the selection subset has the form
\begin{equation}\label{cap4:C}
C(\bc)=\{\xp_1 \in \RR^q  \mid \xp_1 > \bc \}.
\end{equation}
In particular, when $\bc = \zero$, the distribution of $\Y$ is called to be a simple selection distribution.

In this work, we are mainly  interested in the case where $(\X_1,\X_2)$ has a joint density following an arbitrary symmetric multivariate distribution $f_{\X_1,\X_2}$.
For $\Y \overset{d}{=}(\X_2\mid \X_1 \in C)$, this setting leads to a $\Y$ $p$-variate random vector following a skewed version of $f$, which its pdf can be computed in a simpler manner as
\begin{equation}\label{cap4:sel.pdf2}
f_{\Y}(\y) = \frac{
\int_{C} f_{\X_1,\X_2}(\xp_1,\y) \, \dr \xp_1}{\int_{C} f_{\X_1}(\xp_1)\,\dr \xp_1}.
\end{equation}

\subsection{Selection elliptical (SE) distributions}

A quite popular family of selection distributions arises when $\X_1$ and $\X_2$ have a joint multivariate elliptically contoured $(EC)$ distribution, as follows:
\begin{equation}\label{cap41:sel.ec}
\X =  \left(\begin{array}{cc}
\X_1 \\
\X_2
\end{array}
\right)
\sim
EC_{q+p}
\left(
\bxi = \left(\begin{array}{cc}
\bxi_{1}\\
\bxi_{2}
\end{array}
\right),
\bOmega = \left(\begin{array}{cc}
\bOmega_{11} & \bOmega_{12} \\
\bOmega_{21} & \bOmega_{22}
\end{array}
\right),
h^{(q+p)}
\right),
\end{equation}
where $\bxi_{1} \in \RR^q$ and $\bxi_{2} \in \RR^p$ are location vectors, $\bOmega_{11} \in \RR^{q \times q}$, $\bOmega_{22} \in \RR^{p \times p}$, and $\bOmega_{21} \in \RR^{p \times q}$ are dispersion matrices, and, in addition to these parameters, $h^{(q+p)}$ is a density generator function. We denote the selection distribution resulting from \eqref{cap41:sel.ec} by $SLCT\text{-}EC_{p,q}(\bxi,\bOmega,h^{(q+p)},C)$. They typically result in skew-elliptical distributions, except for two cases: $\bOmega_{21} = \zero_{p \times q}$ and $C = C(\bxi_{1})$ (for more details, see \cite{arellano2006unified}). Given that the elliptical family of distributions is closed under marginalization and conditioning, the distribution of $\X_2$ and $(\X_1\mid \X_2=\xp)$ are also elliptical, where their respective pdfs are given by
\begin{align}
\X_2 &\sim EC_p(\bxi_{2},\bOmega_{22},h^{(p)}), \label{cap4:ec.marg&cond1}\\
\X_1\mid \X_2=\x &\sim EC_q(\bxi_{1}+\bOmega_{12} \bOmega_{22}^{-1}(\x - \bxi_{2}),\bOmega_{11}-\bOmega_{12} \bOmega_{22}^{-1}\bOmega_{21},h^{(q)}_{\x}),\label{cap4:ec.marg&cond2}
\end{align}
with induced conditional generator
$$
h_{\xp}^{(q)}(u) = \frac{h^{(q+p)}(u + \delta_{2}(\x))}{h^{(p)}\delta_{2}(\x)},
$$
with $\delta_{2}(\x) \eqdef
(\x - \bxi_{2})^\top \bOmega_{22}^{-1} (\x - \bxi_{2})$.
These last equations imply that the selection elliptical distributions are also closed under marginalization and conditioning. Furthermore, it is well-know that the SE family is closed under linear transformations. For $\A \in \mathbb{R}^{r\times p}$ and $\bp \in \mathbb{R}^r$ being a matrix of rank $r \leq p$ and a vector, respectively, it holds that the linear transformation $\A\Y + \bp \eqdist (\A\X_2 + \bp)\mid(\X_1 > \zero)$, where $\eqdist$ is an acronym that stands for identically distributed, and then
\begin{equation}\label{cap4:AY+b}
\A\Y + \bp
\sim
SLCT\text{-}EC_{r,q}
\left(
\bxi = \left(\begin{array}{cc}
\bxi_{1}\\
\A\bxi_{2} + \bp
\end{array}
\right),
\bOmega = \left(\begin{array}{cc}
\bOmega_{11} & \bOmega_{12} \A^\top \\
\A\bOmega_{21} & \A\bOmega_{22}\A^\top
\end{array}
\right),
h^{(q+r)}
\right).
\end{equation}

Notice from Equation \eqref{cap4:sel.pdf2}, that alternatively we can write

\begin{equation}\label{cap4:ec.sel.pdf2}
f_{\Y}(\y) = \frac{
\int_{C} f_{q+p}(\xp_1,\y;\bxi,\bOmega,h^{(q+p)}) \, \dr \xp_1}{\int_{C} f_{q}(\xp_1;\bxi_{1},\bOmega_{11},h^{(q)})\,\dr \xp_1}.
\end{equation}

\subsection{Particular cases for the SE distribution}

Some particular cases, useful for our purposes, are detailed next. For further details, we refer to \cite{arellano2006unified}.

\subsection*{Unified-skew elliptical (SUE) distribution}

Let $\Y \sim SLCT\text{-}EC_{p,q}(\bxi,\bOmega,h^{(q+p)},C)$.  $\Y$ is said to follow the unified skew-elliptical distribution introduced by \cite{arellano2006unification} when the truncation subset $C= C(\zero)$. From \eqref{cap4:ec.sel.pdf2}, it follows that

\begin{equation}\label{cap4:sue.pdf}
f_{\Y}(\y) =
f_{p}(\y;\bxi_{2},\bOmega_{22},h^{(p)})
\frac{
F_{q}(\bxi_{1} + \bOmega_{12}\bOmega_{22}^{-1}(\y - \bxi_{2});\zero,
\bOmega_{11} - \bOmega_{12} \bOmega_{22}^{-1}\bOmega_{21}
,h^{(q)}_{\y})}
{F_{q}(\bxi_{1};\bOmega_{11},h^{(q)})},
\end{equation}
where $
f_{p}(\y;\bxi_{2},\bOmega_{22},h^{(p)}) = |\bOmega_{22}|^{-1/2}h^{(p)}(\delta_{\X_2}(\y)),
$
and $F_q(\z;\zero,\bTheta,g^{(q)})$ denote the cumulative distribution function (cdf) of the $EC_q(\zero,\bTheta,g^{(q)})$. Note that the density in (\ref{cap4:sue.pdf}) extends the family of skew elliptical distributions proposed by \cite{BrancoD2001} (see also, \cite{AzzaliniC2003}), which consider $q=1$ and $\xi_{1}=0$.

\subsection*{Scale-mixture of unified-skew normal (SMSUN) distribution}

Let $W$ being a nonnegative random variable with cdf $G$. For a generator function $h^{(p+q)}(u)= \int_{0}^{\infty}(2\pi \kappa(w))^{-(p+q)/2}e^{-u/2\kappa(w)} \dr G(w)$, several skewed and thick-tailed distributions can be obtained from different specifications of the weight function $\kappa(\cdot)$ and $G$. It is said that $\Y$ follows a SMSUN distribution, if its probability density function (pdf) takes the general form
\begin{equation}\label{cap4:sm.sue.pdf}
f_{\Y}(\y) = \int_0^\infty
\phi_{p}(\y;\bxi_{2},\kappa(w)\bOmega_{22})
\frac{
\Phi_{q}(\bxi_{1} + \bOmega_{12}\bOmega_{22}^{-1}(\y - \bxi_{2});
\kappa(w)\{\bOmega_{11} - \bOmega_{12} \bOmega_{22}^{-1}\bOmega_{21}\})}
{\Phi_{q}(\bxi_{1};\kappa(w)\bOmega_{11})} \dr G(w),
\end{equation}

where $\Phi_r(\cdot;\bSigma)$ represents the cdf of a $r$-variate normal distribution with mean vector $\zero$ and variance-covariance matrix $\bSigma$. Here $\Y\mid(W=w)$ follow a unified skew-normal (SUN) distribution, where we write $\Y\mid(W=w) \sim SUN(\bxi,\kappa(w)\bOmega)$.

\begin{itemize}
\item \subsection*{Unified skew-normal (SUN) distribution}

Setting $W$ as a degenerated r.v. in 1 ($\PP(W = 1)=1$) and $\kappa(w) = w$, then $h^{(p+q)}(u)=(2\pi)^{-(p+q)/2}e^{-u/2}$, $u\geq 0$, for which
$h^{(p)}(u) = (2\pi)^{-p/2} e^{-u/2}$. Then, $\Y$ follow a SUN distribution, that is, $\Y \sim SUN_{p,q}(\bxi,\bOmega)$, with pdf as

\begin{equation}\label{cap4:sun.pdf}
f_{\Y}(\y) =
\phi_{p}(\y;\bxi_{2},\bOmega_{22})
\frac{
\Phi_{q}(\bxi_{1} + \bOmega_{12}\bOmega_{22}^{-1}(\y - \bxi_{2});
{\bOmega_{11} - \bOmega_{12} \bOmega_{22}^{-1}\bOmega_{21}})}
{\Phi_{q}(\bxi_{1};\bOmega_{11})}.
\end{equation}

\item \subsection*{Unified skew-$t$ (SUT) distribution}

For $W \sim G(\nu/2,\nu/2)$ and weight function $\kappa(w) = 1/w$, we obtain
$h^{(p+q)}(u)=\displaystyle\tfrac{\Gamma((p+q+\nu)/2)\nu^{\nu/2}}{\Gamma(\nu/2)
\pi^{(p+q)/2}}\{1+u\}^{-(p+q+\nu)/2}$ and hence \eqref{cap4:sm.sue.pdf} becomes

\begin{equation}\label{cap4:sut.pdf}
f_{\Y}(\y) =
t_{p}(\y;\bxi_{2},\bOmega_{2},\nu)
\frac{
T_{q}(\bxi_{1} + \bOmega_{12}\bOmega_{22}^{-1}(\y - \bxi_{2});
\displaystyle\frac{\nu + \delta_2(\y)}{\nu + p}
\{\bOmega_{11} - \bOmega_{12} \bOmega_{22}^{-1}\bOmega_{21}\},\nu+p)}
{T_{q}(\bxi_{1};\bOmega_{11},\nu)},
\end{equation}

where $T_r(\cdot;\bSigma,\nu)$ represents the cdf of a $r$-variate Student's $t$ distribution with location vector $\zero$, scale matrix $\bSigma$ and degrees of freedom $\nu$. For $\Y$ with pdf as in \eqref{cap4:sut.pdf} is said to follow a SUT distribution, which is denoted by $\Y \sim SUT_{p,q}(\bxi,\bOmega,\nu)$ and was introduced by \cite{arellano2006unification}. It is well-know that \eqref{cap4:sut.pdf} reduces to a SUN pdf \eqref{cap4:sun.pdf} as $\nu \rightarrow \infty$ and to an unified skew-Cauchy (SUC) distribution, when $\nu = 1$.

Furthermore, using the following parametrization:
\begin{equation}\label{cap4:xiomega}
\bxi = \left(\begin{array}{cc}
\btau\\
\bmu
\end{array}
\right)
\qquad
\text{and}
\qquad
\bOmega = \left(\begin{array}{cc}
\bPsi + \bLambda^\top\bLambda & \bOmega_{12} \\
\bOmega_{21} & \bSigma
\end{array}
\right),
\end{equation}
where $\bOmega_{21}=\bSigma^{1/2}\bLambda$, with $\bSigma^{1/2}$ being the square root matrix  of $\bSigma$ such that $\bSigma = \bSigma^{1/2}\bSigma^{1/2}$,
we use the notation $\Y \sim SUT_{p,q}(\bmu,\bSigma,\bLambda,\btau,\nu,\bPsi)$, to stand for a $p$-variate EST distribution with location parameter $\bmu \in \RR^p$, positive-definite scale matrix $\bSigma\in\RR^{p\times p}$, shape matrix parameter $\blambda\in\RR^{p\times q}$, extension vector parameter $\btau\in \RR^q$ and positive-definite correlation matrix $\bPsi\in\RR^{q\times q}$. The pdf $\Y$ is now simplified to
\begin{equation}\label{cap4:sut.pdf.2}
SUT_{p,q}(\mathbf{y};\bmu,\bSigma,\bLambda,\btau,\nu,\bPsi)= t_p(\mathbf{y};\bmu,\bSigma,\nu)
\frac{
T_q
\big(
(
\btau +
\bLambda^{\top}\bSigma^{-1/2}(\mathbf{y}-\bmu)
)
\,
\nu(\y)
,\bPsi
;\nu + p
\big)
}{T_q(\btau;\bPsi+\bLambda^\top\bLambda,\nu)},
\end{equation}
with $\nu^2(\xp) \equiv \nu_\X^2(\x) \eqdef {(\nu+dim(\x))/(\nu+\delta(\x))}$ and $\delta(\x) = (\x - \bmu_\X)^\top \bSigma_\X^{-1}(\x - \bmu_\X)$ being the Mahalanobis distance. The pdf in \eqref{cap4:sut.pdf.2} is equivalent to the one found in \cite{arellano2010multivariate}, with a different parametrization. Although the unified skew-$t$ distribution above is appealing from a theoretical point of view, the particular case, when $q=1$, leads to simpler but flexible enough distribution of interest for practical purposes.

\subsection*{Extended skew-$t$ (EST) distribution}

For $q=1$, we have that $\bPsi = 1$, $\bLambda = \blambda$ and $T_q(\xp;\bPsi,\nu) = T_1(x/\sqrt{\psi},\nu)$, hence \eqref{cap4:sut.pdf.2} reduces to the pdf of a EST distribution, denoted by $EST_p(\mathbf{y};\bmu,\bSigma,\blambda,\tau)$, that is,
\begin{equation}\label{cap4:est.pdf}
EST_p(\mathbf{y};\bmu,\bSigma,\blambda,\tau)= t_p(\mathbf{y};\bmu,\bSigma,\nu)
\frac{
T_1
\big(
(
\tau +
\blambda^{\top}\bSigma^{-1/2}(\mathbf{y}-\bmu)
)
\nu(\y)
;\nu + p
\big)
}{T_1(\tautil;\nu)}.
\end{equation}
with $\tautil = \tau/\sqrt{1+\blambda^\top\blambda}$ .Here, $\blambda \in \RR^p$ is a shape parameter which regulates the skewness of $\Y$, and $\tau \in \RR$ is a scalar. Location and scale parameters $\bmu$ and $\bSigma$ remains as before. Here, we write $\Y\sim EST_p(\bmu,\bSigma,\blambda,\tau)$ Notice that, $SUT_{p,1}\equiv EST_p$.
Besides, it is straightforward to see that
$$  EST_p(\y;\bmu,\bSigma,\blambda,\tau,\nu) {\longrightarrow} \,t_p(\y;\bmu,\bSigma,\nu),\,\,{\text as }\,\,\,\tau\rightarrow \infty,$$
where $t_p(\cdot;\bmu,\bSigma,\nu)$ corresponds to the pdf of a multivariate Student's $t$ distribution with location parameter $\bmu$, scale parameter $\bSigma$ and degrees of freedom $\nu$.
On the other hand, when $\tau=0$, we retrieve the skew-$t$ distribution $ST_p(\bmu,\bSigma,\blambda,\nu)$ say, which density function is given by
\begin{equation}\label{cap4:st.pdf}
ST_p(\y;\bmu,\bSigma,\blambda,\nu)= 2{t_p(\y;\bmu,\bSigma,\nu)
\,
T_1
\big(
\blambda^{\top}\bSigma^{-1/2}(\mathbf{y}-\bmu)
\,\nu(\y)
;\nu + p
\big)
},
\end{equation}
that is, $ EST_p(\bmu,\bSigma,\blambda,0,\nu)=
ST_p(\bmu,\bSigma,\blambda,\nu)$. Further properties were studied in \cite{arellano2010multivariate}, but with a slightly different parametrization.

\begin{figure}[!htb]
\centering
\includegraphics[width=0.47\textwidth]{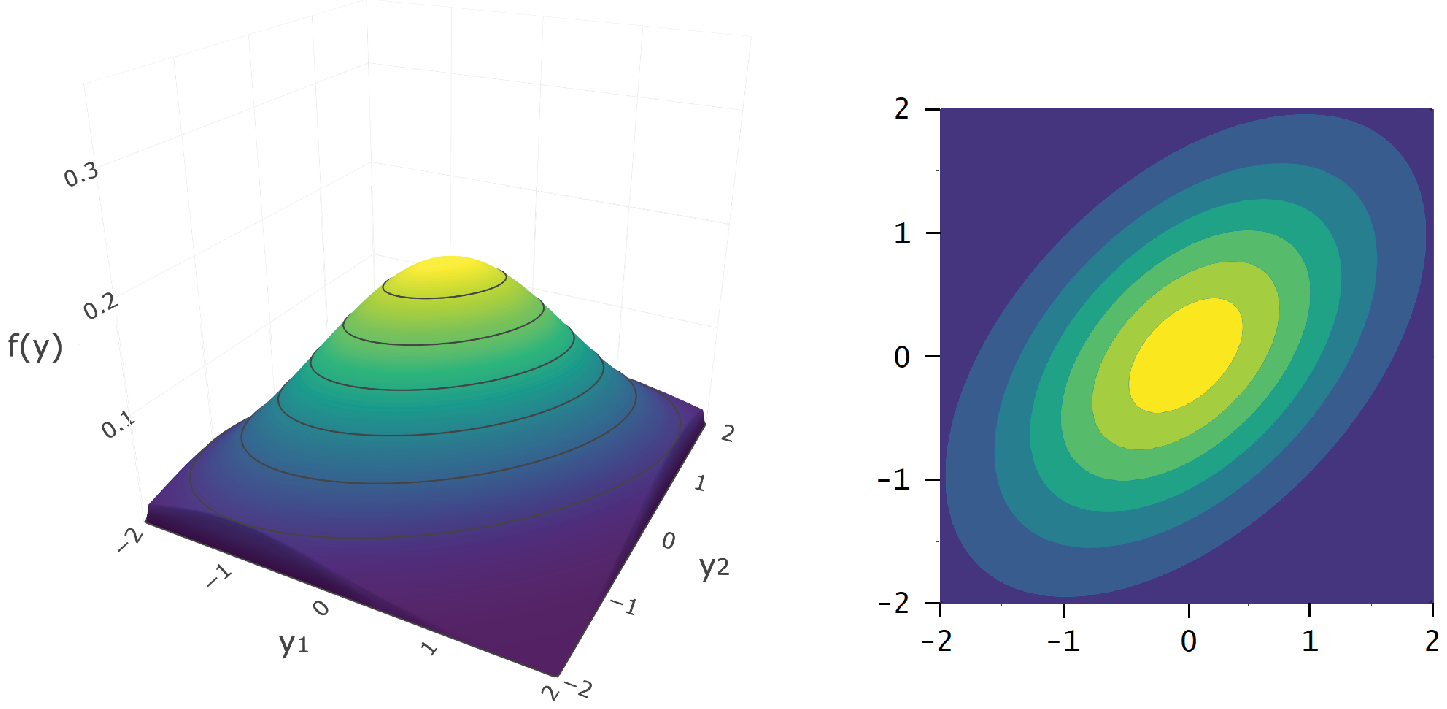}
\hspace{0.4cm}
\includegraphics[width=0.47\textwidth]{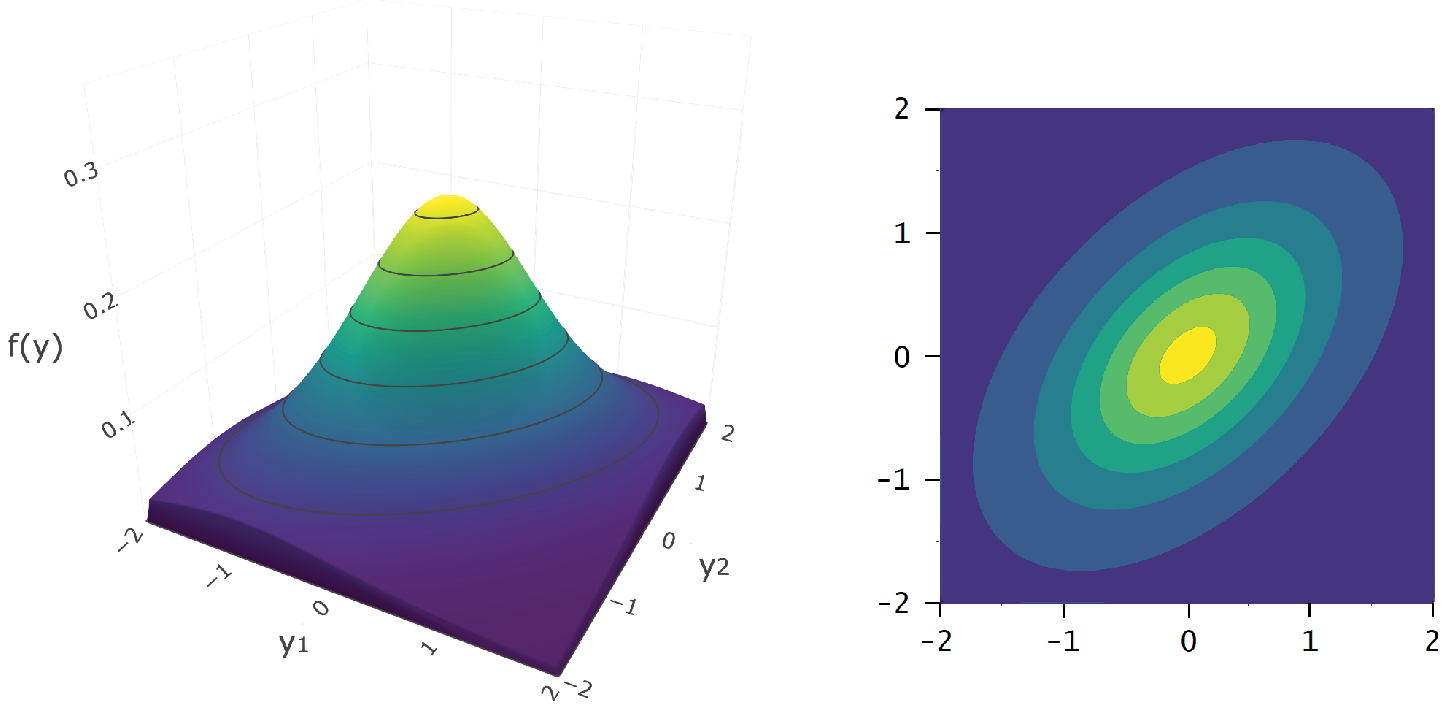}
\includegraphics[width=0.47\textwidth]{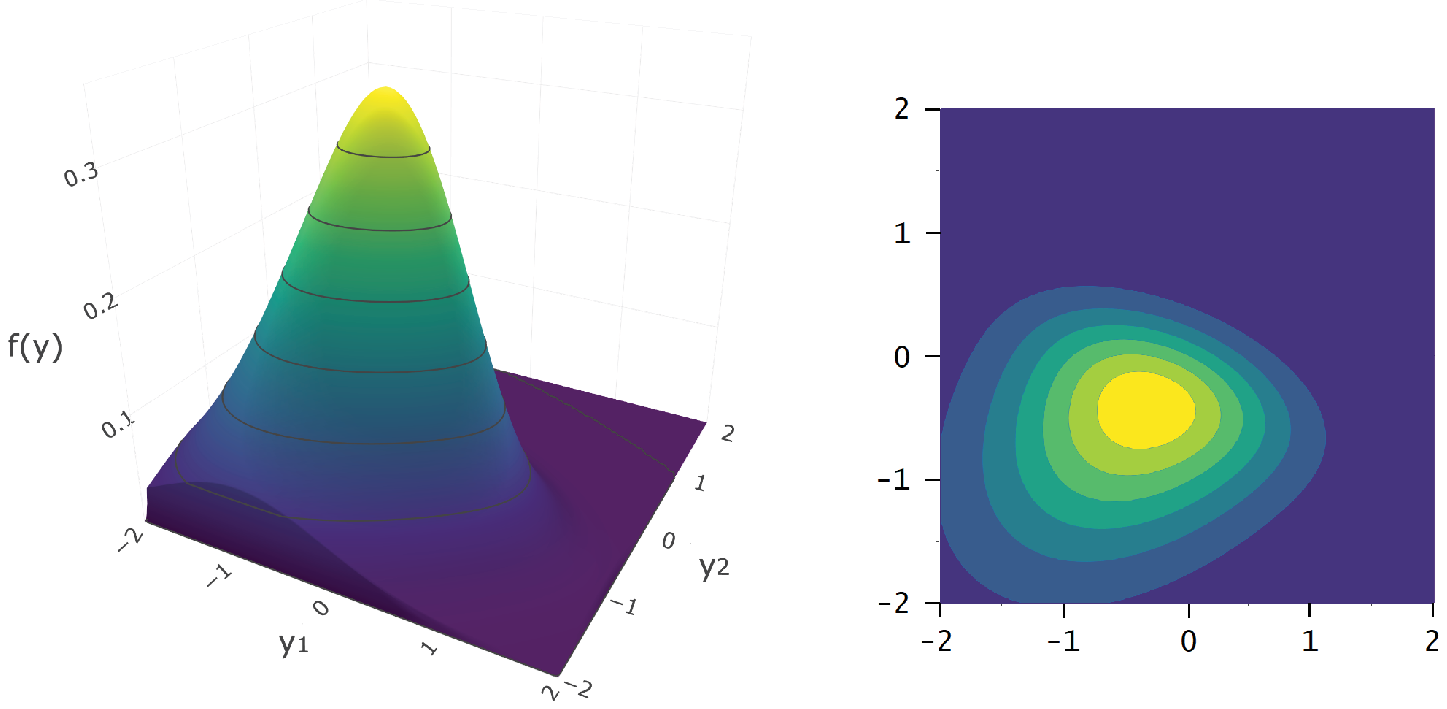}
\hspace{0.4cm}
\includegraphics[width=0.47\textwidth]{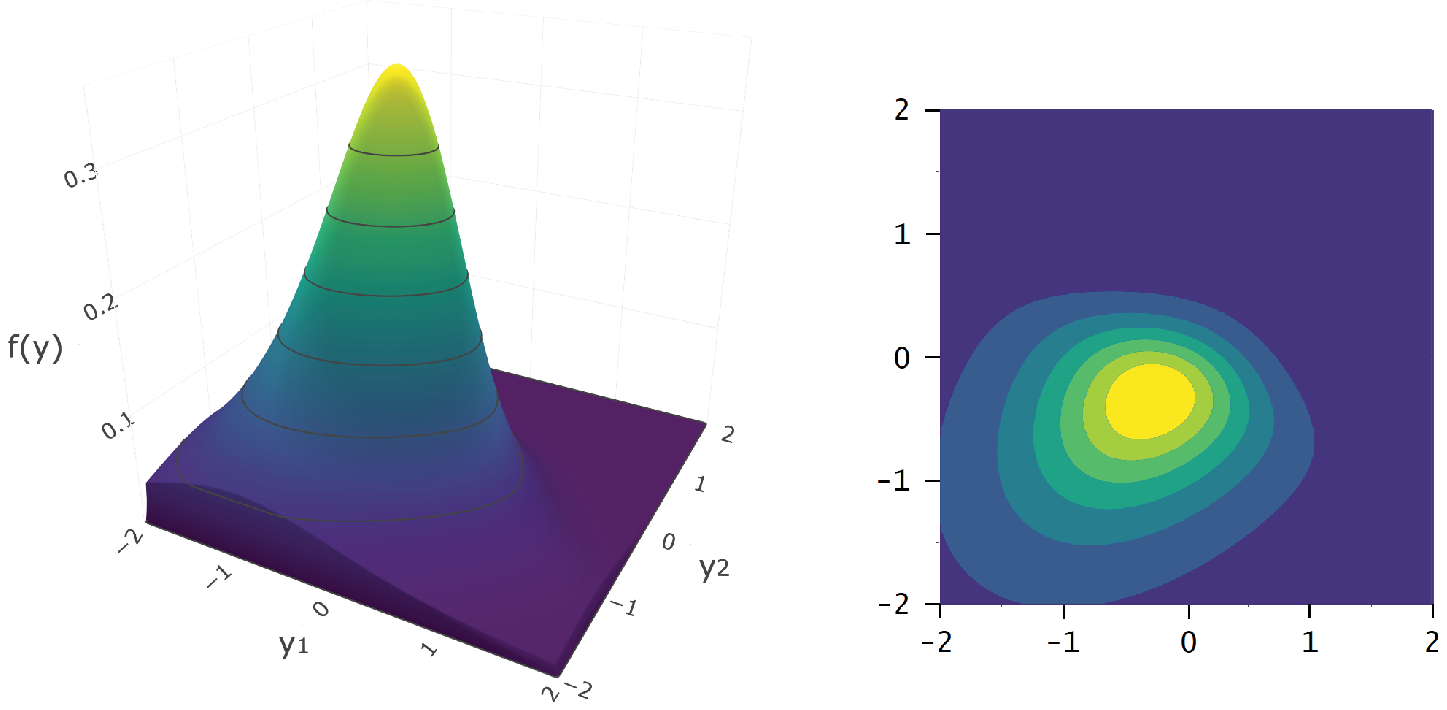}
\includegraphics[width=0.47\textwidth]{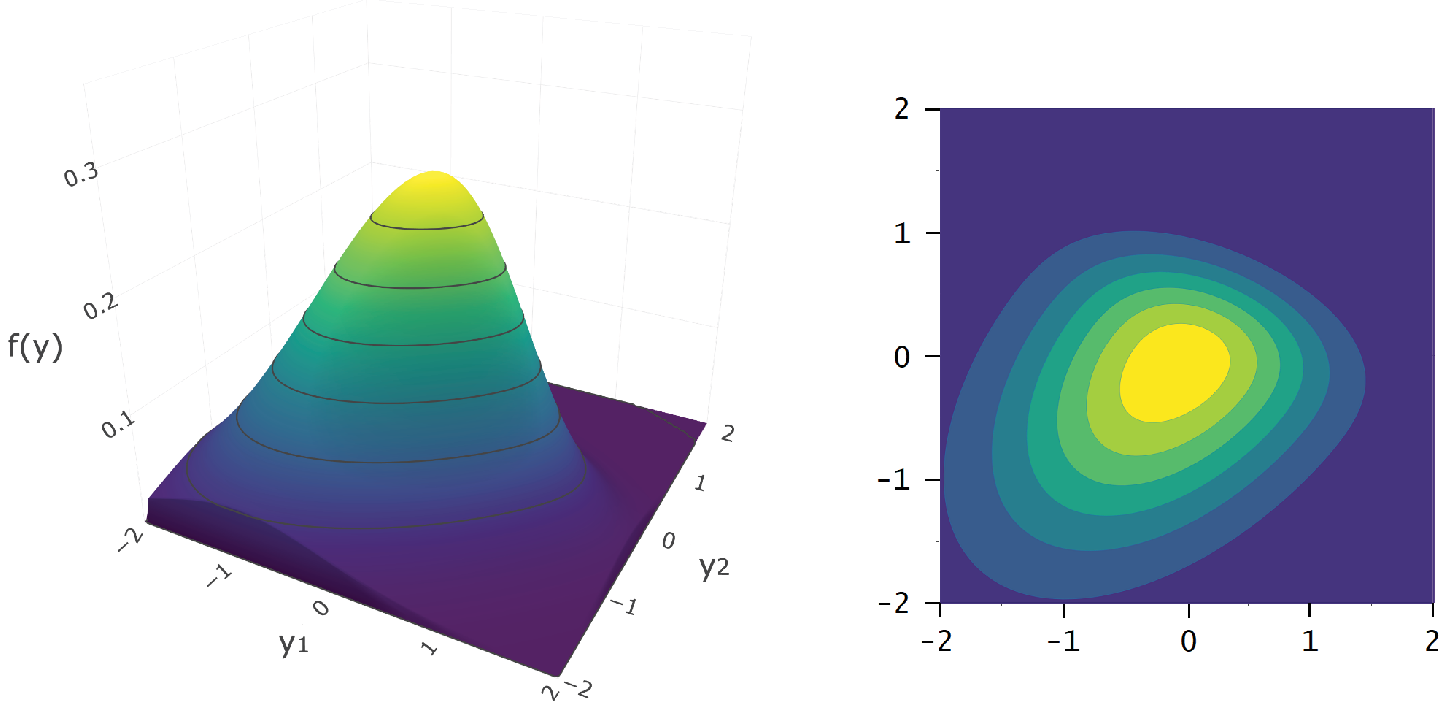}
\hspace{0.4cm}
\includegraphics[width=0.47\textwidth]{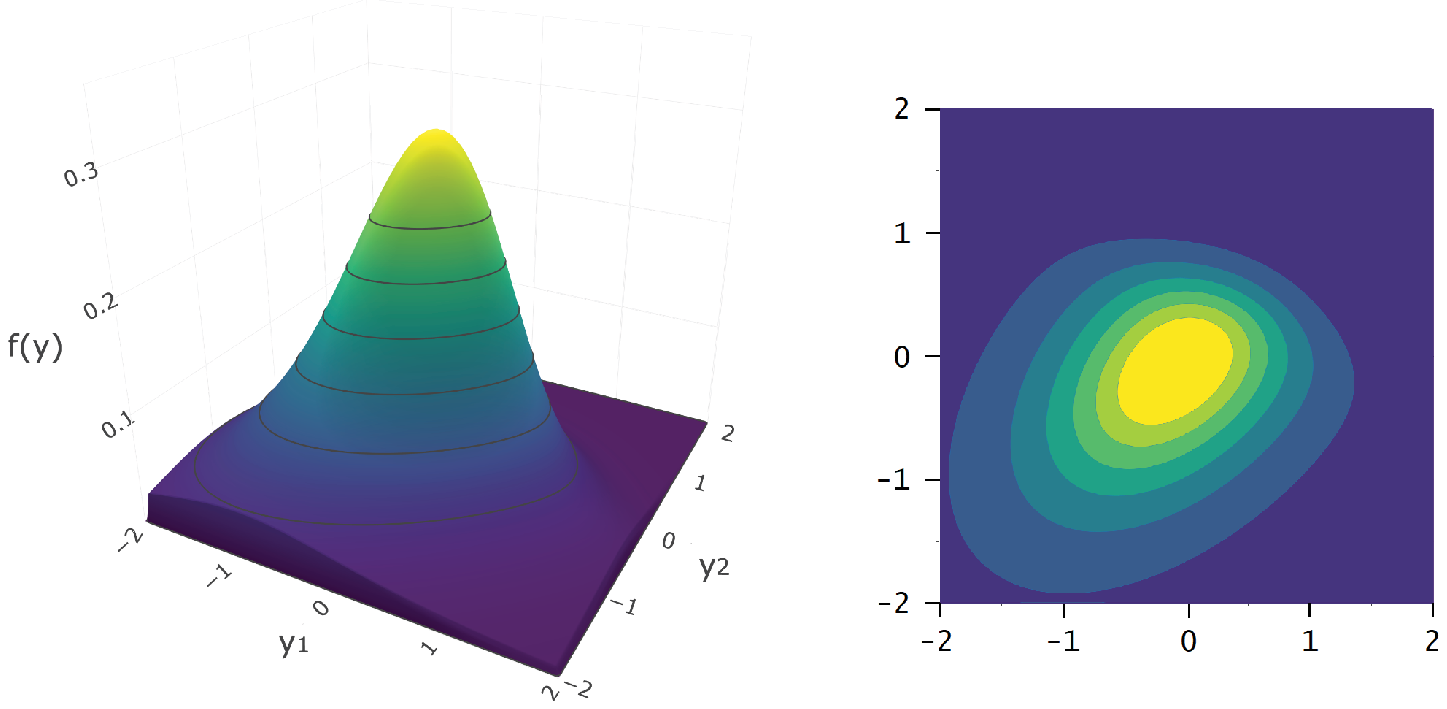}
\caption{Densities for particular cases of a truncated SUT distribution. Normal cases at left column (normal, SN and ESN from top to bottom) and Student's-$t$ cases at right (Student's $t$, ST and EST from top to bottom).}
\label{cap4:fig0}
\end{figure}

Six different densities for special cases of the truncated SUT distribution are shown in Figure~\ref{cap4:fig0}. Symmetrical cases normal and Student's $t$ are shown at first row ($\blambda = \zero$), skew cases: skew-normal (SN) and ST at second row ($\tau = 0$) and extended skew cases: extended skew-normal (ESN) and EST at the third row. Location vector $\bmu$ and scale matrix $\bSigma$ remains fixed for all cases.

\item \subsection*{Others unified skewed distributions}

Others unified members are given by different combinations of the weight function $\kappa(W)$ and the mixture cdf $G$. For instance, we obtain an \emph{unified skew-slash} distribution when $\kappa(w)=1/w$ and $W\sim\mathrm{Beta}(\nu,1)$; an \emph{unified skew-contaminated-normal} distribution when $\kappa(W)=1/W$ and $W$ is a discrete r.v. with probability mass function (pmf) $g(w;\nu,\gamma)=\nu{\mathbb{I}}_{\{w=\gamma\}}+(1-\nu){\mathbb{I}}_{\{w=1\}}$, with $\mathbb{I}$ being the identity function.
Besides, \cite{BrancoD2001} mentions some other distributions as the skew-logistic, skew-stable, skew-exponential power, skew-Pearson type II and finite mixture of skew-normal distribution. It is worth mentioning that even though \cite{BrancoD2001} works with a subclass of the SMSUN, when $q=1$ and $\xi_{1}=0$, unified versions of these are readily computed by considering the same respective weight function $\kappa(\cdot)$ and mixture distribution $G$.

\end{itemize}

\section{On moments of the doubly truncated selection elliptical distribution}\label{cap4:moments}

Let $\Y \sim SLCT\text{-}EC_{p,q}(\bxi,\bOmega,h^{(q+p)},C)$ with pdf as defined in \eqref{cap4:ec.sel.pdf2} and let also $\mathbb{A}$ be a Borel set in $\mathbb{R}^p$. We say that a random vector $\W$ has a truncated selection elliptical (TSE) distribution on $\mathbb{A}$ when $\W \eqdist \Y | (\Y \in \mathbb{A})$. In this case, the pdf of $\W$ is given by
$$
f_\W(\w)=\displaystyle\frac{f_\Y(\w)}{P(\Y \in  \mathbb{A})}\mathbf{1}_{\mathbb{A}}(\w),
$$
where  $\mathbf{1}_{ \mathbb{A}}$ is the indicator function of the set $ \mathbb{A}$. We use the notation $\W \sim TSLCT\text{-}EC_{p,q}(\bxi,\bOmega,h^{(q+p)},$
$C;\mathbb{A})$. If $\mathbb{A}$ has the form
\begin{equation} \label{cap4:hyper}
\mathbb{A} = \{(y_1,\ldots,y_p)\in \mathbb{R}^p:\,\,\, a_1\leq y_1 \leq b_1,\ldots, a_p\leq y_p \leq b_p \}=\{\mathbf{y}\in\mathbb{R}^p:\mathbf{a}\leq\mathbf{y}\leq\mathbf{b}\},
\end{equation}
we say that  the distribution of $\W$ is doubly truncated distribution and we use the notation $\{\Y \in  \mathbb{A}\}=\{\ap \leq \Y \leq \mathbf{b}\}$, where  $\mathbf{a}=(a_1,\ldots,a_p)^\top$ and $\mathbf{b}=(b_1,\ldots,b_p)^\top$, where $a_i$ and $b_i$ values may be infinite, by convention.  Analogously we define $\{\Y \geq \mathbf{a}\}$ and $\{\Y \leq \mathbf{b}\}$. Thus, we say that the distribution of $\W$ is truncated from below and truncated from above, respectively. For convenience, we also use the notation $\W \sim TSLCT\text{-}EC_{p,q}(\bxi,\bOmega,h^{(q+p)},C;(\ap,\bp))$ with the last parameter indicating the truncation interval. Analogously, we do denote $TEC_{p}(\bxi,\bOmega,h^{(p)};(\ap,\bp))$ to refer to a $p$-variate (doubly) truncated elliptical (TE) distribution on $(\ap,\bp)\in\RR^p$. Some characterizations of the doubly TE have been recently discussed in \cite{Moran-vasquez2019}.

\subsection{Moments of a TSE distribution}\label{cap4:moments}

For two $p$-{dimensional vectors}
$\y=(y_1,\ldots,y_p)^{\top}$ and
$\kp=(k_1,\ldots,k_p)^{\top}$,  let $\y^{\kp}$ stand for $(y_1^{k_1},\ldots,y_p^{k_p})$, that is, we use a pointwise notation. Next, we present a formulation to compute arbitrary product moments of a TSLCT-EC distribution.

\begin{theorem}[{\bf moments of a TSE}]\label{cap4:theo.1}
Let $\X \sim EC_{q+p}(\bxi,\bOmega,h^{(q+p)})$
as defined in \eqref{cap4:sel.ec}.  Let
$C$ be a truncation subset of the form $C(\bc,\bd)=\{\xp_1 \in \RR^q  \mid \bc \leq \xp_1 \leq \bd \}$. For $\Y \sim SLCT\text{-}EC_{p,q}(\bxi,\bOmega,h^{(q+p)},C(\bc,\bd))$, then $\EE[\Y^{\kp}] = \EE[Y_1^{k_1} Y_2^{k_2} \ldots Y_p^{k_p}]$ can be computed as
\begin{equation}\label{cap4:kappamoment}
\EE[\Y^{\kp} \mid\ap\leq\Y\leq\bp] = \EE[\X^{\bkappa}\mid\balpha\leq\X\leq\bbeta],
\end{equation}
with $\bkappa = (\zero^\top_q,\kp^\top)^\top$, $\balpha = (\bc^\top,\ap^\top)^\top$ and $\bbeta = (\bd^\top,\bp^\top)^\top$, where $\kp = (k_1,k_2,\ldots,k_p)^\top$, with $k_i \in \mathbb{N}$, for $\ii$.
\end{theorem}
\begin{proof}
Since $\Y \eqdist \X_2\mid(\bc \leq \X_1 \leq \bd)$, the proof is direct by noting that
\begin{align*}
\Y\mid(\ap\leq\Y\leq\bp) &\eqdist \X_2\mid(\bc \leq \X_1 \leq \bd \, \cap \, \ap\leq\X_2\leq\bp) \\
&\eqdist \X_2\mid(\balpha\leq\X\leq\bbeta).
\end{align*}
\end{proof}
\begin{corollary}[{\bf first two moments of a TSE}]\label{cap4:cor.1}
Under the same conditions of Theorem 1, let
$\bmp = \EE[\X \mid \balpha \leq \X \leq \bbeta]$ and $\MM = \EE[\X\X^\top \mid \balpha \leq \X \leq \bbeta]$, both partitioned as
$$
\bmp = \left(\begin{array}{cc}
\bmp_1\\
\bmp_2
\end{array}
\right)
\qquad\text{and}\qquad
\MM = \left(\begin{array}{cc}
\MM_{11} & \MM_{12} \\
\MM_{21} & \MM_{22}
\end{array}
\right),
$$
respectively. Then, the first two moments of $\Y\mid(\ap\leq\Y\leq\bp)$ are given by
\begin{align}
\EE[\Y\mid\ap\leq\Y\leq\bp] &= \bmp_2, \label{cap4:mainres1} \\
\EE[\Y\Y^\top\mid\ap\leq\Y\leq\bp] &= \MM_{22}, \label{cap4:mainres2}
\end{align}
where $\bmp_2 \in \RR^p$ and $\MM_{22} \in \RR^{p \times p}$.
\end{corollary}

For the particular truncation subset $C(\bc)$ as in \eqref{cap4:C}, Theorem 1 and Corollary 1 hold considering $\balpha = (\bc^\top,\ap^\top)^\top$ and $\bbeta = (\binfty^\top,\bp^\top)^\top$. Notice that, Theorem \ref{cap4:theo.1} and Corollary \ref{cap4:cor.1} state that we are able to compute any arbitrary moment of $\Y\mid(\ap\leq\Y\leq\bp)$, that is, a TSE distribution just using an unique corresponding moment of a doubly TE distribution $\X\mid(\balpha\leq\X\leq\bbeta)$.

This is highly convenient since doubly truncated moments for some members of the elliptical family of distributions are already available in the literature and statistical softwares. In particular for the truncated multivariate normal and Student's-t we have the R packages {\it TTmoment}, {\it tmvtnorm} and {\it MomTrunc}.

\subsection{Dealing with limiting and extreme cases}

Consider
$\X\sim EC_{q+p}(\bxi,\bOmega,h^{(q+p)})$ and $\Y \sim SLCT\text{-}EC_{p,q}(\bxi,\bOmega,h^{(q+p)},C)$ as in Theorem \ref{cap4:theo.1} with truncation subset $C = C(\zero)$.
As $\bxi_1 \rightarrow \binfty$, we have that $\PP(\X_1 \geq \zero) \rightarrow 1$. Besides, as $\bxi_1 \rightarrow -\binfty$, we have that $\PP(\X_1 \geq \zero) \rightarrow 0$ and consequently $\PP(\ap \leq \Y \leq \bp) = \PP(\balpha \leq \X \leq \bbeta)/\PP(\X_1 \geq \zero)\rightarrow \infty$.
Thus, for $\bxi_1$ containing high negative values small enough, sometimes we are not able to compute $\mathbbm{E}[\Y^\kp]$ due to computation precision, mainly when we work with distributions with lighter tails densities. For instance, for a normal univariate case, $\Phi_1(\xi_1)=0$ for $\xi_1 \leq -38$ in \texttt{R} software.  The next proposition helps us to circumvent this problem.

\begin{proposition}[{\bf limiting case of a SE}]\label{cap4:prop:lambdas0}
\noindent As $\bxi_1\rightarrow -\binfty$, i.e., $\xi_{1i} \rightarrow -\infty,\, i = 1,\ldots,q$, then
\begin{equation}\label{cap4:slct.ec.tau_inf_neg0}
SLCT\text{-}EC_{p,q}(\bxi,\bOmega,h^{(q+p)},C(\zero)) {\longrightarrow} EC_p(\bxi_{2} - \bOmega_{21} \bOmega_{11}^{-1}\bxi_{1},\bOmega_{22}-\bOmega_{21}\bOmega_{11}^{-1}\bOmega_{12},h^{(p)}_{\zero}).
\end{equation}
\end{proposition}

\begin{proof}
Let
$\X=(\X_1^\top,\X_2^\top)^\top \sim EC_{q+p}(\bxi,\bOmega,h^{(q+p)})$ and
$\Y \sim TSLCT\text{-}EC_{p,q}(\bxi,\bOmega,h^{(q+p)},$
$C(\zero);(\ap,\bp))$. As $\bxi_1 \rightarrow -\binfty$, we have that $\PP(\X_1\geq\zero)\rightarrow 0$, $\mathbbm{E}[\X_1|\X_1\geq\zero]\rightarrow\zero$ and $\mathrm{var}[\X_1|\X_1\geq\zero] \rightarrow \zero$, hence $\X_1|\X_1 \geq \zero$ becomes degenerated on $\zero$. From Definition \ref{cap4:def1}, $\Y \stackrel{d}{\longrightarrow} (\X_2|\X_1=\zero)$, and by the conditional distribution in Equation \eqref{cap4:ec.marg&cond2}, it is straightforward to show that $\X_2|\X_1 \sim EC_p(\bxi_{2} + \bOmega_{21} \bOmega_{11}^{-1}(\X_1 - \bxi_{1}),\bOmega_{22}-\bOmega_{21}\bOmega_{11}^{-1}\bOmega_{12},h^{(p)}_{\X_1})$. Evaluating $\X_1=\zero$ we achieve \eqref{cap4:slct.ec.tau_inf_neg0} concluding the proof.
\end{proof}

\subsection{Approximating the mean and variance-covariance of a TE distribution for extreme cases}

While using the relation \eqref{cap4:mainres1} and \eqref{cap4:mainres2}, we may face numerical problems trying to compute $\bmp = \EE[\X \mid \balpha \leq \X \leq \bbeta]$ and $\MM = \EE[\X\X^\top \mid \balpha \leq \X \leq \bbeta]$ for extreme settings of $\bxi$ and $\bOmega$. Usually, it occurs when $\PP(\balpha \leq \X \leq \bbeta)\approx 0$ because the probability density is far from the integration region $(\balpha,\bbeta)$. It is worth mentioning that, for these cases, it is not even possible to estimate the moments generating Monte Carlo (MC) samples via rejection sample due to the high rejection ratio when subsetting to a small integration region. Other methods as Gibbs sampling are preferable under this situation.

Hence, we present correction method in order to approximate the mean and the variance-covariance of a multivariate TE distribution even when the numerical precision of the software is a limitation.

\subsubsection{Dealing with out-of-bounds limits}

Consider $\X \sim EC_{r}\big(\bxi,\bOmega,h^{(r)}\big)$ to be partitioned as $\X = (\X_1^{T},\X_2^\top)^\top$ such that $dim(\X_1) =
r_1$, $dim(\X_2) = r_2$, \CG{where} $r_1 + r_2 = r$. Also, consider $\bxi$, $\bOmega$, $\balpha = (\balpha_1^\top,\balpha_2^\top)^\top$ and $\bbeta = (\bbeta_1^\top,\bbeta_2^\top)^\top$ partitioned as before. Suppose that we are not able to compute $\EE[\X^{\bkappa}|\balpha \leq \X \leq \bbeta]$, because there exists a partition $\X_2$ of $\X$ of dimension $r_2$ that is out-of-bounds, that is $P(\balpha_2\leq \X_2 \leq \bbeta_2) \approx 0$. Notice that this happens because $\PP(\balpha \leq \X \leq \bbeta) \leq P(\balpha_2\leq \X_2 \leq \bbeta_2) \approx 0.$  Besides, we suppose that $P(\balpha_1\leq \X_1 \leq \bbeta_1) > 0$. Since the limits of $\X_2$ are out-of-bounds (and $\balpha_2 < \bbeta_2$), we have two possible cases: $\bbeta_2 \rightarrow -\binfty$ or $\balpha_2 \rightarrow \binfty$. For convenience, let $\bmu_2 = \EE[\X_2\mid \balpha_2\leq \X_2 \leq \bbeta_2]$ and $\bSigma_{22} = \mathrm{cov}[\X_2\mid \balpha_2\leq \X_2 \leq \bbeta_2]$. For the first case, as $\bbeta_2\rightarrow -\binfty$, we have that $\bmu_2\rightarrow\bbeta_2$ and $\bSigma_{22}\rightarrow \zero_{r_2\times r_2}$. Analogously, we have that $\bmu_2\rightarrow\balpha_2$ and $\bSigma_{22}\rightarrow \zero_{r_2\times r_2}$ as $\balpha_2\rightarrow \binfty$. Hence, $\X_2\mid (\balpha_2\leq \X_2 \leq \bbeta_2)$ is degenerated on $\bmu_2$ and then $\X_{1.2} \eqdist \X_1 \mid (\X_2 = \bmu_2) \sim EC_{r_1}(\bxi_{1}+\bOmega_{12} \bOmega_{22}^{-1}(\bmu_2 - \bxi_{2}),\bOmega_{11}-\bOmega_{12} \bOmega_{22}^{-1}\bOmega_{21},h^{(r_1)}_{\bmu_2})$. Given that $\mathrm{cov}[\EE[\X_1|\X_2]] = \zero$ and
$\mathrm{cov}[\EE[\X_1|\X_2],\X_2] = \zero$, it follows that
\begin{equation}\label{cap4:cond oob}
\EE[\X\mid \balpha \leq \X \leq \bbeta] = \left[\hspace{-2mm}
\begin{array}{c}
\bmu_{1.2}
\\
\bmu_2
\end{array}
\hspace{-2mm}
\right]
\qquad \text{and} \qquad
\mathrm{cov}[\X \mid \balpha \leq \X \leq \bbeta] =
\left[\begin{array}{cc}
\bSigma_{11.2}&
\zero_{r_1\times r_2} \\
\zero_{r_2\times r_1} &
\zero_{r_2\times r_2}
\end{array}
\right],
\end{equation}

\noindent with $\bmu_{1.2} = \EE[\X_{1.2} \mid \balpha_1 \leq \X_{1.2} \leq \bbeta_1]$ and $\bSigma_{11.2} = \mathrm{cov}[\X_{1.2} \mid \balpha_1 \leq \X_{1.2} \leq \bbeta_1]$ being the mean and variance-covariance matrix of a $r_1$-variate TE distribution.\\

In the event that there are double infinite limits, we can part the vector as well, in order to avoid unnecessary calculation of these integrals.

\subsubsection{Dealing with double infinite limits}\label{cap4:sec:doubinfs}

Now, consider $\X = (\X_1^\top,\X_2^\top)^\top$ to be partitioned such that the upper and lower truncation limits associated with $\X_1$ are both infinite, but at least one of the truncation limits associated with $\X_2$ is finite. Then $r_1$ be the number of pairs in $(\balpha,\bbeta)$ that are both infinite, that is, $dim(\X_1)=r_1$ and $dim(\X_2)=r_2$, by complement. Since $\balpha_1=-\binfty$ and
$\bbeta_1=\binfty$
, it follows that
$\X_2 \mid (\balpha \leq \X \leq \bbeta) \sim {TEC}_{r_2}\big(\bxi_2,\bOmega_{22},h^{(r_2)};[\balpha_2,\bbeta_2]\big)$
and
$\X_1|\X_2 \sim EC_{r_1}\big(\bxi_1 + \bOmega_{12}\bOmega_{22}^{-1}(\X_2 - \bxi_2),\bOmega_{11} - \bOmega_{12}\bOmega_{22}^{-1}\bOmega_{21},h^{(r_1)}_{\X_2}\big)$. Let $\bmu_2 = \EE[\X_2\mid \balpha_2 \leq \X_2 \leq \bbeta_2]$ and $\bSigma_{22} = \cov[\X_2\mid \balpha_2 \leq \X_2 \leq \bbeta_2]$. Hence, it follows that $\EE[\X\mid \balpha \leq \X \leq \bbeta] = \EE[\EE[\X_1 \mid \X_2] \mid \balpha_2 \leq \X_2 \leq \bbeta_2]$, that is
\begin{align}\label{cap4:cond_inf_mean}
\EE[\X\mid \balpha \leq \X \leq \bbeta] &= \EE
\left.
\left[
\left(
\begin{array}{c}
\bxi_1 + \bOmega_{12}\bOmega_{22}^{-1}(\X_2 - \bxi_2)
\nonumber
\\
\X_2
\end{array}
\hspace{-2mm}
\right)
\right|
\balpha_2 \leq \X_2 \leq \bbeta_2
\right]\\
&= \left[\hspace{-2mm}
\begin{array}{c}
\bxi_1 + \bOmega_{12}\bOmega_{22}^{-1}(\bmu_2 - \bxi_2)
\\
\bmu_2
\end{array}
\hspace{-2mm}
\right].
\end{align}
On the other hand, we have that $\mathrm{cov}[\X_2,\EE[\X_1|\X_2]] = \mathrm{cov}[\X_2,\X_2\bOmega_{22}^{-1}\bOmega_{21}] = \bSigma_{22}\bOmega_{22}^{-1}\bOmega_{21}$, $\mathrm{cov}[\EE[\X_1|\X_2]] =
\bOmega_{12}\bOmega_{22}^{-1}\bSigma_{22}\bOmega_{22}^{-1}\bOmega_{21}
$ and  $\EE[\mathrm{cov}[\X_1|\X_2]] = \omega_{1.2}(\bOmega_{11} - \bOmega_{12}\bOmega_{22}^{-1}\bOmega_{21})$, with $\omega_{1.2}$ being a constant depending of the conditional generating function $h^{(r_1)}_{\X_2}$. Finally,
\begin{align}\label{cap4:cond_inf_var}
\mathrm{cov}[\X\mid \balpha \leq \X \leq \bbeta]
&=
\left[\begin{array}{cc}
\omega_{1.2}\bOmega_{11} -  \bOmega_{12}\bOmega_{22}^{-1}\big(\omega_{1.2}\bI_{p_2} - \bSigma_{22}\bOmega_{22}^{-1}\big)\bOmega_{21}&
\bOmega_{12}\bOmega_{22}^{-1}\bSigma_{22} \\
\bSigma_{22}\bOmega_{22}^{-1}\bOmega_{21} &
\bSigma_{22}
\end{array}
\right],
\end{align}
where $\bmu_2$ and $\bSigma_{22}$ are the mean vector and variance-covariance matrix of a TE distribution, so we can use \eqref{cap4:mainres1} and \eqref{cap4:mainres2} as well.
{
\begin{remark}
Note that $\X_1 \mid (\balpha \leq \X \leq \bbeta)$ does not follow a non-truncated elliptical distribution, that is, $\X_1 \mid (\balpha \leq \X \leq \bbeta)
\nsim EC_{r_1}\big(\bxi_1,\bOmega_{11},h^{(r_1)}\big)$ even though $-\binfty \leq \X_1 \leq \binfty$. This occurs due to
$
\X_1 \mid (\balpha \leq \X \leq \bbeta)  =
\X_1 \mid (\balpha_2 \leq \X_2 \leq \bbeta_2)
$
. In general, the marginal distributions of a TE distribution are not TE, however this holds for $\X_2$ due to the particular case $\balpha_1 = -\binfty$ and $\bbeta_1 = \binfty$.
\end{remark}}

\subsubsection*{Particular cases}

Notice that the constant $\omega_{1.2}$ will vary depending of the elliptical distribution we are using. For instance, if $\X \sim t_{r_1+r_2}(\bxi,\bOmega,\nu)$ then
it follows that
$\X_2
\sim {t}_{r_2}\big(\bxi_2,\bOmega_{22},\nu\big)$ and $\X_1|\X_2 \sim t_{r_1}\big(\bxi_1 + \bOmega_{12}\bOmega_{22}^{-1}(\X_2 - \bxi_2),(\bOmega_{11} - \bOmega_{12}\bOmega_{22}^{-1}\bOmega_{21})/\nu^2(\X_2)
,\nu+r_2\big)$. In this case, it takes the form $\omega_{1.2} = \EE[(\nu+r_2)/(\nu+r_2-2)\nu^2(\X_2)\mid \balpha_2 \leq \X_2 \leq \bbeta_2]$, which is given by
\begin{align}\label{cap4:omega21}
\omega_{1.2} &=
\EE
\left[
\frac{\nu+\delta(\X_2)}{\nu+r_2-2}\mid \balpha_2 \leq \X_2 \leq \bbeta_2
\right],
\nonumber\\
&=
\left(
\frac{\nu}{\nu-2}
\right)
\frac{L_{r_2}(\balpha_2,\bbeta_2;\bxi_2,\nu\bOmega_{22}/(\nu-2),\nu-2)}
{L_{r_2}(\balpha_2,\bbeta_2;\bxi_2,\bOmega_{22},\nu)},
\end{align}

\noindent
where $L_r(\balpha,\bbeta;\bxi,\bOmega,\nu)$ denotes the integral
\begin{equation}\label{cap4:LT}
L_r(\balpha,\bbeta;\bxi,\bOmega,\nu) = \int_{\balpha}^{\bbeta}{{{t}_r}(\y;\bxi,\bOmega,\nu)\textrm{d}\y},
\end{equation}
that is, $L_r(\balpha,\bbeta;\bxi,\bOmega,\nu)=\PP(\balpha \leq \Y \leq \bbeta)$ for $\Y \sim t_r(\bxi,\bOmega,\nu)$. Probabilities in \eqref{cap4:omega21} are involved in the calculation of $\bmu_2$ and $\bSigma_{22}$ so they are recycled. For the normal case, it is straightforward to see that $\omega_{1.2} = 1$, by taking $\nu \rightarrow \infty$.

\indent As can be seen, we can use equations \eqref{cap4:cond_inf_mean} and \eqref{cap4:cond_inf_var} to deal with double infinite limits, where the truncated moments are computed only over a $r_2$-variate partition, avoiding some unnecessary integrals and saving significant computational effort. On the other hand, expression \eqref{cap4:cond oob} let us to approximate the mean and the variance-covariance matrix for cases where the computational precision is a limitation.

\subsection{Existence of the moments of a TE and TSE distribution}\label{cap4:existence}

It is well know that for some members of EC family of distributions, their moments do not exist, however, this could be different depending of the truncation limits.

Let $\X \sim EC_r(\bxi,\bOmega,h^{(r)})$ be partitioned as in Subsection \ref{cap4:sec:doubinfs}, with $r_1$ being the number of pairs in $(\balpha,\bbeta)$ that are both finite and $r_2 = r-r_1$. Similarly, $\bkappa = (\bkappa_1^\top,\bkappa_2^\top)^\top$ is partitioned as well.
If $r_1 = r$, then the truncation limits $\balpha$ and $\bbeta$ contains only finite elements, and hence $\EE[\X^{\bkappa} \mid\balpha\leq\X\leq\bbeta]$ exists for all $\bkappa\in\mathbb{N}^r$ because the distribution is bounded. When $r_2 \geq 1$, there exists at least one pair in $(\balpha,\bbeta)$ containing infinite values, and the expectation may not exist. Given that $\EE[\X^{\bkappa} \mid\balpha\leq\X\leq\bbeta] = \EE[\X_1^{\bkappa_1}\EE[\X_2^{\bkappa_2} \mid\X_1,\balpha_2\leq\X_2\leq\bbeta_2]\mid\balpha_1\leq\X_1\leq\bbeta_1]$,
for any measurable function $g$,
$\EE[g(\X_1)\mid\balpha_1\leq\X_1\leq\bbeta_1]$ always exists,
and $(\balpha_2,\bbeta_2)$ is not bounded,
it is straightforward to see that $\EE[\X^{\bkappa} \mid\balpha\leq\X\leq\bbeta]$ exist if and only if (\emph{iff}) the inner expectation $\EE[\X_2^{\bkappa_2} \mid\X_1]$
exists.

As seen, the existence only depends of the order of the moment $\bkappa_2$ and the distribution of $\X_2|\X_1$, this last depending on the conditional generating function $h_{\X_1}^{(r_2)}$.

If $\Y \sim SLCT\text{-}EC_{p,q}(\bxi,\bOmega,h^{(q+p)},C)$, with truncation subset of the form $C(\bc,\bd)$ and $r = p+q$ say. It follows from Theorem \ref{cap4:theo.1}, that $\EE[\Y^{\kp} \mid\ap\leq\Y\leq\bp] = \EE[\X^{\bkappa} \mid\balpha\leq\X\leq\bbeta]$. Hence, the same condition holds taking in account that $\bkappa = (\zero^\top_q,\kp^\top)^\top$, $\balpha = (\bc^\top,\ap^\top)^\top$ and $\bbeta = (\bd^\top,\bp^\top)^\top$. Next, we present a result for a particular case.

\section{The doubly truncated SUT distribution}

For the rest of the paper we shall focus our attention on the computation of the moments of the doubly truncated unified skew-$t$ (TSUT) distribution, denoted by $\W \sim TSUT_{p,q}(\bmu,\bSigma,\bLambda,\btau,\nu,\bSigma;(\ap,\bp))$. Besides, we shall study some of its properties and for its particular case (when $q=1$), the doubly truncated extended skew-$t$ distribution, say $\W\sim TEST_p(\bmu,\bSigma,\blambda,\tau,\nu;(\ap,\bp))$. For the limiting symmetrical case, we shall use the notation $\W\sim Tt_p(\bmu,\bSigma,\nu;(\ap,\bp))$ to refer to a $p$-variate truncated Student-$t$ (TT) distribution on $(\ap,\bp)\in \RR^p$.
Finally, $\W \sim TN_p(\bmu,\bSigma;(\ap,\bp))$ will stand for a $p$-variate truncated normal distribution on the interval $(\ap,\bp)$
. Hereinafter we shall omit the expression \emph{doubly} due to we only work with intervalar truncation.\\

\begin{corollary}[{\bf moments of a TSUT}]\label{cap4:cor:expectSUT_to_T}
If $\Y \sim SUT_{p,q}(\bmu,\bSigma,\bLambda,\btau,\nu,\bPsi)$, it follows from Theorem \ref{cap4:theo.1} that
$$
\mathbb{E}[\Y^{\kp}\mid\ap \leq \Y \leq \mathbf{b}]=\mathbb{E}[\X^{\bkappa}\mid\balpha\leq \X \leq \bbeta],
$$
where $\X \sim t_{q+p}(\bxi,\bOmega,\nu)$ with $\bxi$ and $\bOmega$ as defined in Equation \eqref{cap4:xiomega} and $\bkappa = (\zero^\top_q,\kp^\top)^\top$, $\balpha = (\zero_q^\top,\ap^\top)^\top$ and $\bbeta = (\binfty_q^\top,\bp^\top)^\top$.
\end{corollary}

\subsection{Mean and covariance matrix of the TSUT distribution}\label{cap4:subsec:meanvar}

Let $\Y \sim TSUT_{p,q}(\bmu,\bSigma,\bLambda,\btau,\nu,\bPsi;(\ap,\bp))$ and $\X \sim Tt_{q+p}(\bxi,\bOmega,\nu;(\balpha,\bbeta))$. From Corollary \ref{cap4:cor:expectSUT_to_T}, we have that the first two moments of $\Y$ can be computed as
\begin{align}
\mathbb{E}[\Y] &= \bmp_2,\label{cap4:eq:MEANVAR_EST1}\\
\mathbb{E}[\Y\Y^\top] &= \MM_{22},\label{cap4:eq:MEANVAR_EST2}
\end{align}
where
$\bmp = \EE[\X]$ and $\MM = \EE[\X\X^\top]$ are partitioned as in Corollary \ref{cap4:cor.1}. Notice that $\mathrm{cov}[\Y] = \mathbb{E}[\Y\Y^\top] - \mathbb{E}[\Y]\mathbb{E}[\Y^\top]$.

Equations \eqref{cap4:eq:MEANVAR_EST1} and \eqref{cap4:eq:MEANVAR_EST2} are convenient for computing $\EE[\Y]$ and $\mathrm{cov}[\Y]$ since all boils down to compute the mean and the variance-covariance matrix for a $q+p$-variate TT distribution which can be calculated using the our \texttt{MomTrunc R} package available on CRAN.

\subsection*{Existence of the moments of a TSUT}

Let also $p_1$ be the number of pairs in $(\ap,\bp)$ that are both finite. Without loss of generality, we assume $\Y=(\Y_1^\top,\Y_2^\top)^\top$, where the upper and lower truncation limits associated with $\Y_1$ are both finite, but at least one of the truncation limits associated with $\Y_2$ is not finite, say $dim(\Y_1) = p_1$ and $dim(\Y_2) = p_2$, with $p_1+p_2 = p$. Consider the partitions of $\ap=(\ap_1^\top,\ap_2^\top)^\top$ ,$\bp=(\bp_1^\top,\bp_2^\top)^\top$ and $\kp=(\kp_1^\top,\kp_2^\top)^\top$ as well. The next proposition gives a sufficient condition for the existence of the moment of a TSUT distribution.

\begin{proposition}[{\bf existence of the moments of a TSUT}]\label{cap4:existenceSUT}
Under the conditions above, $\EE[\Y^{\kp} \mid\ap\leq\Y\leq\bp]$ exists \emph{iff} $sum(\kp_2)< \nu+p_1$.
\end{proposition}

\begin{proof}
From subsection \ref{cap4:existence}, it is suffices to demonstrate that $\EE[\X_2^{\bkappa_2}|\X_1]$ exists. Since $\balpha = (\zero^\top_q,\ap_1^\top,\ap_2^\top)^\top$ and $\bbeta = (\binfty^\top_q,\bp_1^\top,\bp_2^\top)^\top$, it follows that $r_1 = p_1$, $r_2=q+p_2$, $\bkappa_1 = \kp_1$ and $\bkappa_2 = (\zero_q^\top,\kp_2^\top)^\top$. It is easy to show that the distribution of $\X_2|\X_1$ is a $(q+p_2)$-variate Student-$t$ distribution with $\nu+p_1$ degrees of freedom. Hence, the above expectation exists \emph{iff} $sum(\kp_2)<\nu+p_1$.
\end{proof}


From Proposition \ref{cap4:existenceSUT}, see that $\EE[\Y]$ and $\EE[\Y\Y^\top]$ exist \emph{iff} $\nu + p_1 > 1$ and $\nu + p_1 > 2$ respectively. 

\begin{remark}[Sufficient condition of existence of the first two moments of a TSUT]
Since $\nu>0$, it is equivalent to say that, the first moment exists if at least one dimension containing a finite limit exists. Besides, the second moment exists if at least two dimensions containing a finite limit exist.
\end{remark}

\begin{remark}
Sufficient conditions aforementioned hold for the truncated Student-$t$ ($q=0$) and for the truncated EST distribution ($q=1$) due to the condition does not depend on $q$.
\end{remark}

Next, in light of proposition \ref{cap4:prop:lambdas0}, we propose a corollary for the limiting case of a SUT pdf when $\btau\rightarrow -\binfty$.


\begin{corollary}\label{cap4:corprop:lambdas}
\noindent Under the condition of Proposition 1, as $\btau\rightarrow -\binfty$, i.e., $\tau_i \rightarrow -\infty,\, i = 1,\ldots,q$, then
\begin{equation}\label{cap4:sut.tau_inf_neg0}
SUT_{p,q}(\bmu,\bSigma,\bLambda,\btau,\nu,\bPsi) {\longrightarrow} t_p(\bgamma,\omega_{\tau}\bGamma,\nu+q),
\end{equation}
with $\bgamma = \bmu - \bOmega_{21}\bOmega_{11}^{-1}\btau$, $\bGamma = \bSigma - \bOmega_{21}\bOmega_{11}^{-1}\bOmega_{12}$ and
$\omega_\tau = \nu_{\scriptscriptstyle\X_1}^2(\zero) = (\nu + \btau^\top\bOmega_{11}^{-1}\btau)/(\nu+q)$
with $\bOmega_{11} = \bPsi + \bLambda^\top\bLambda$.

In particular, for $q=1$,
\begin{equation}\label{cap4:est.tau_inf_neg0}
EST_p(\bmu,\bSigma,\blambda,\tau,\nu) {\longrightarrow} t_p(\bgamma,(\nu + \tilde\tau^2)/(\nu+1)\bGamma,\nu+1),
\end{equation}
with $\bgamma = \bmu - \tilde\tau
\bDelta$, 
$\bGamma = \bSigma - \bDelta\bDelta^{\top}$, and 
$\bDelta = 
\bSigma^{1/2}\blambda/\sqrt{1+\blambda^\top\blambda}
$.
\end{corollary}

It is worth to stress that parameters $\bDelta$ and $\bGamma$ are well know in the context of SN and ST modeling since they are used in the the stochastic representation of this variates. Furthermore, the resulting symmetric distribution is highly involved in the framework of censored modeling as shown next in Section 6.

\section{Numerical example}

\begin{figure*}[!t]
\centering

\includegraphics[width=0.3\textwidth]{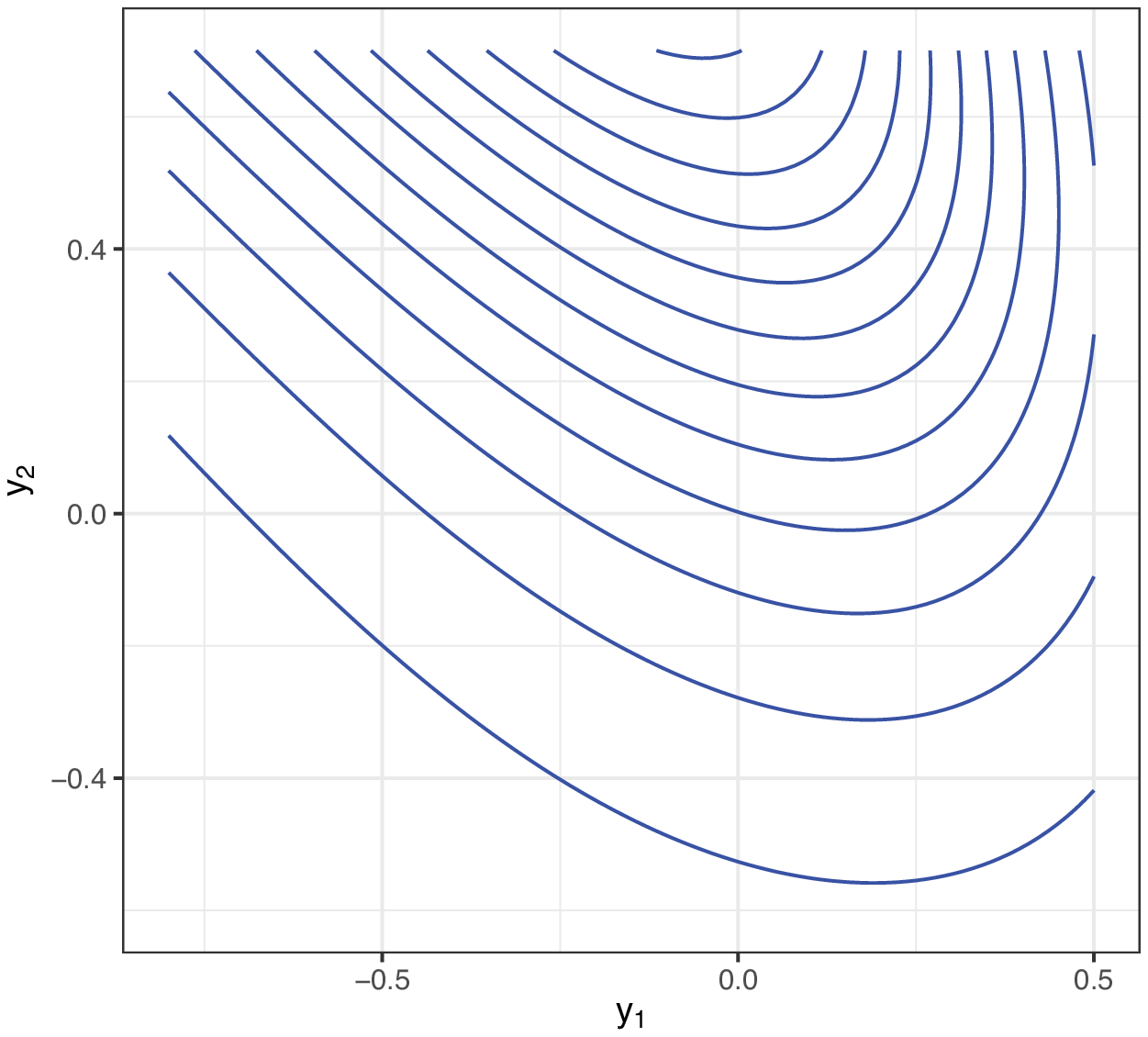}
\hspace{0.1cm}
\includegraphics[width=0.3\textwidth]{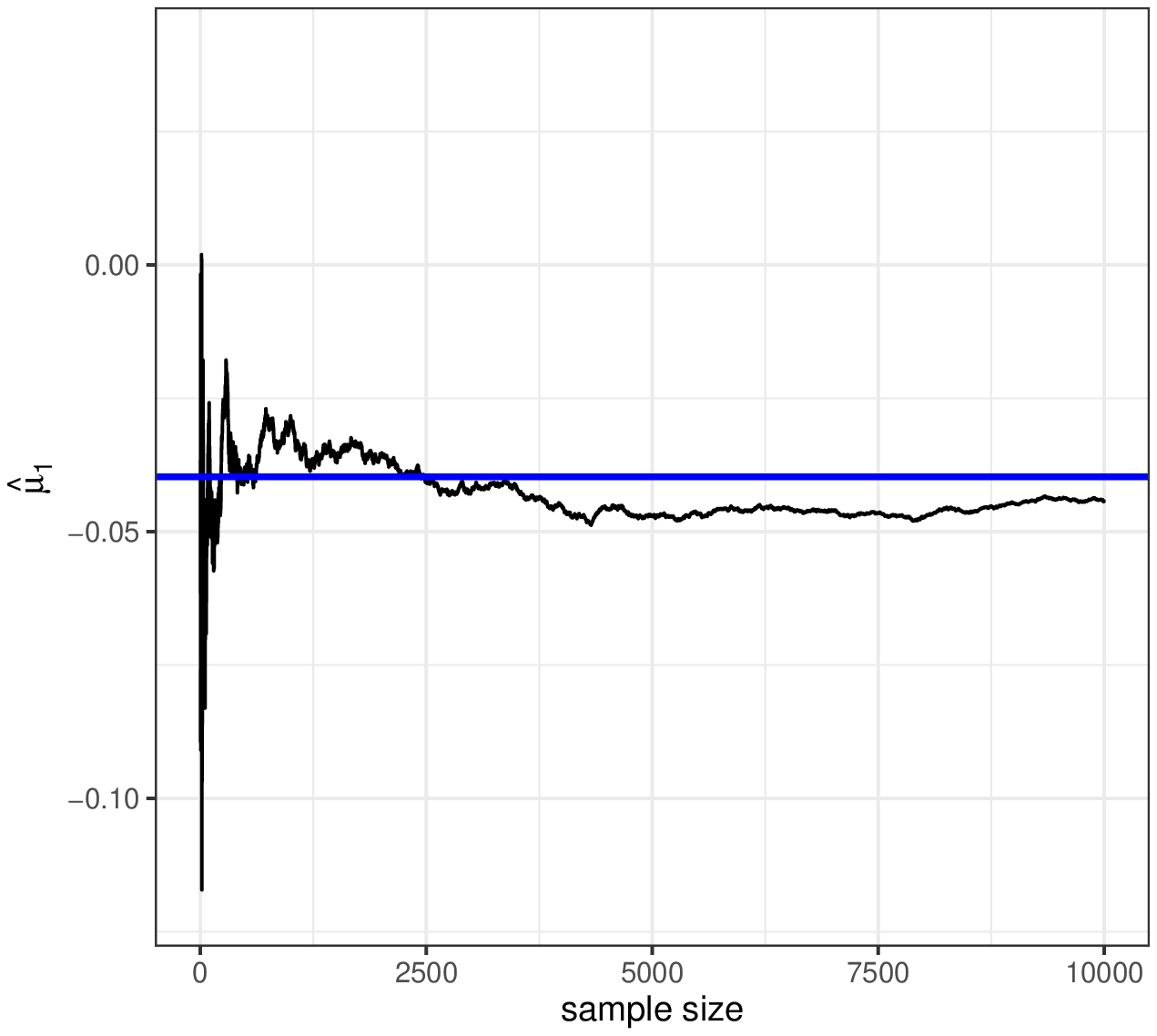}
\hspace{0.1cm}
\includegraphics[width=0.3\textwidth]{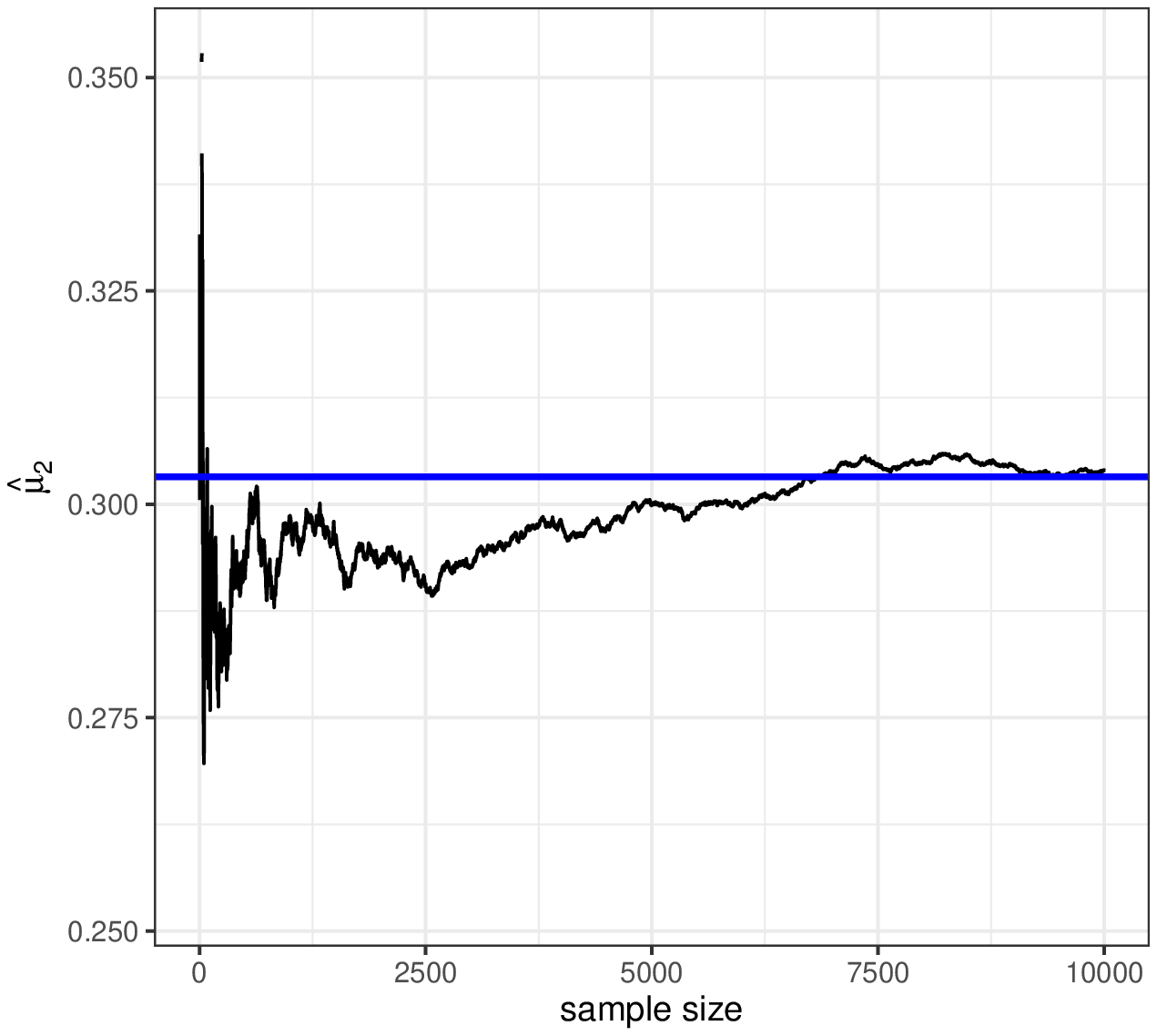}

\vskip15pt

\includegraphics[width=0.3\textwidth]{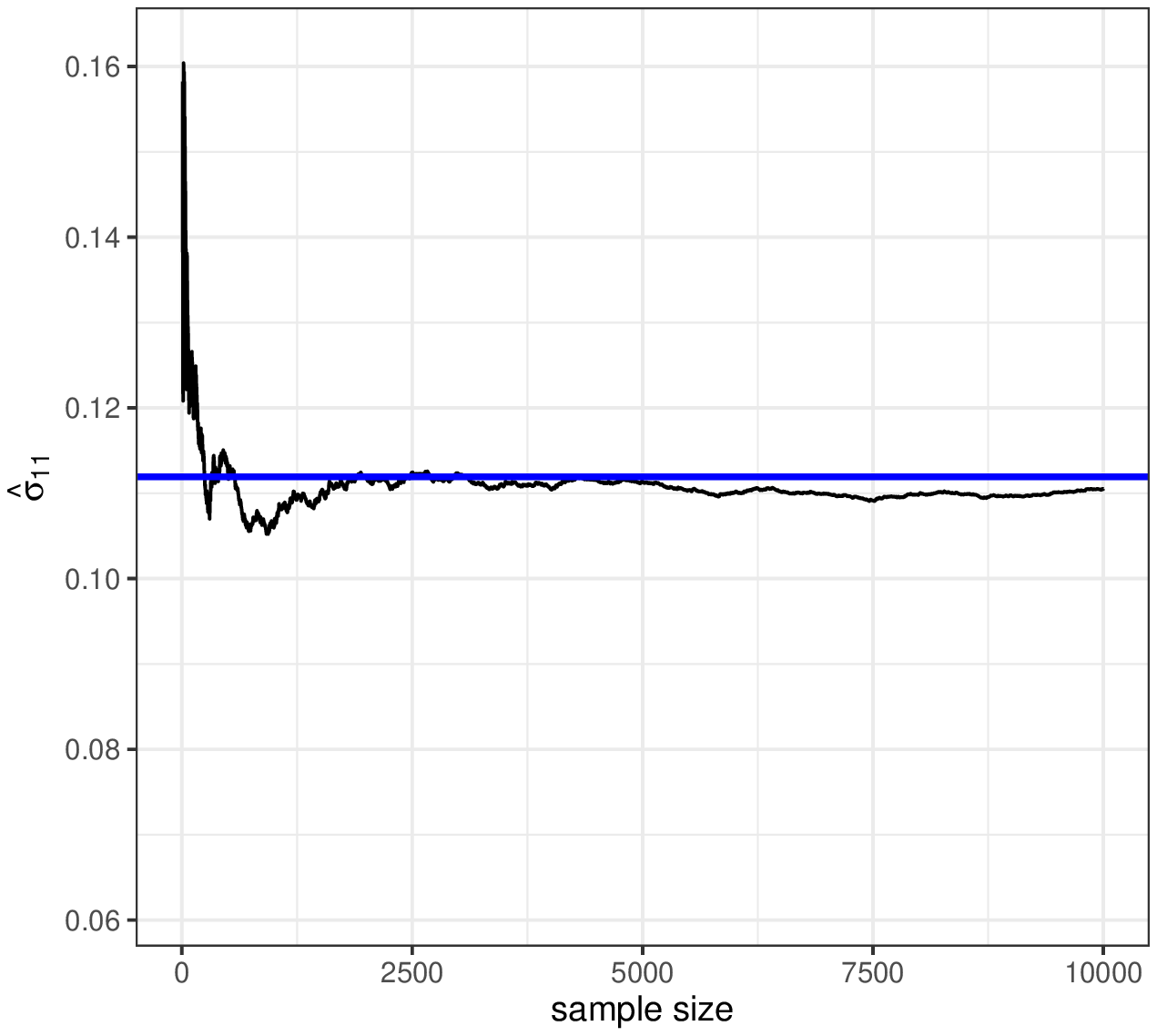}
\hspace{0.1cm}
\includegraphics[width=0.3\textwidth]{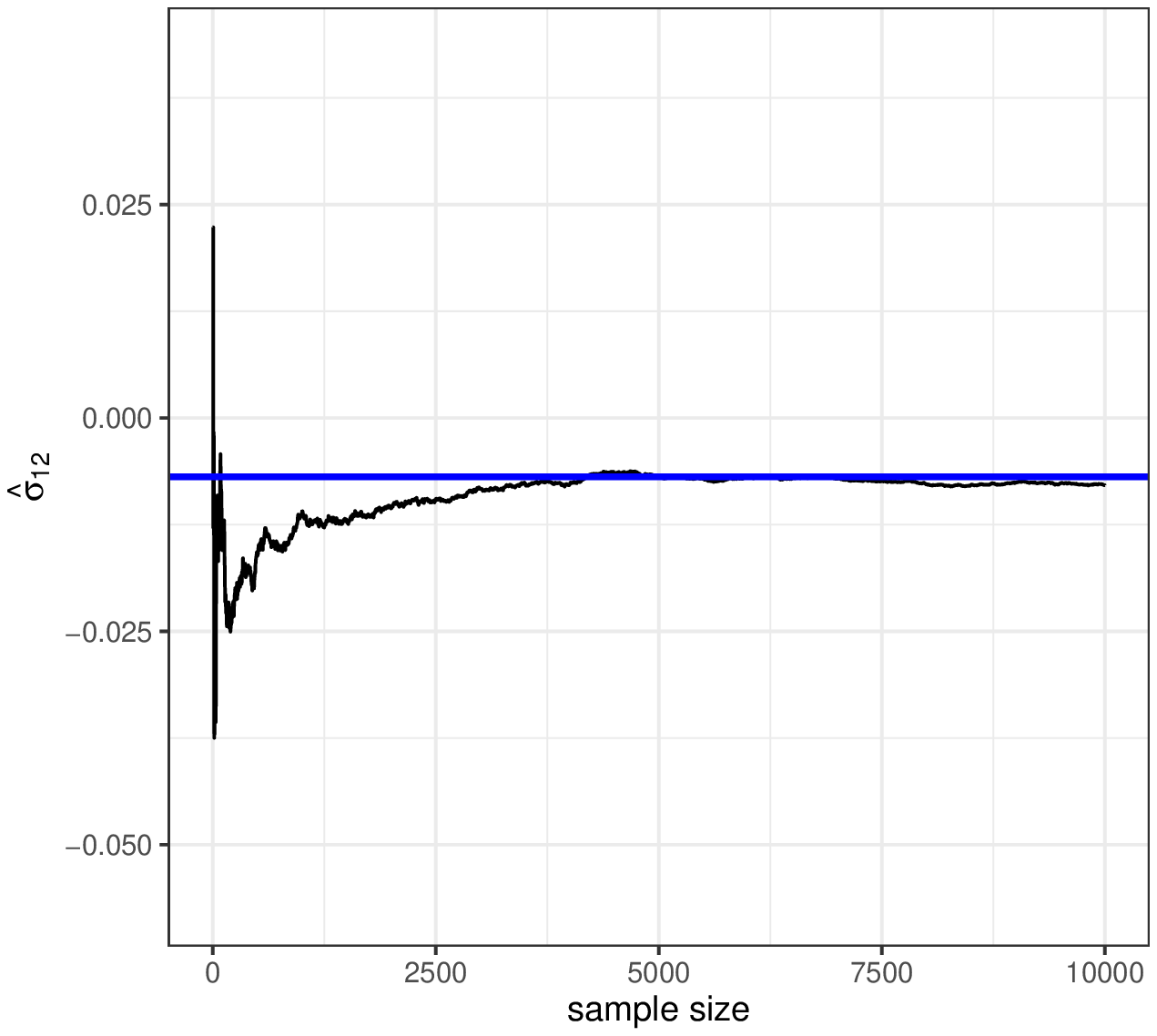}
\hspace{0.1cm}
\includegraphics[width=0.3\textwidth]{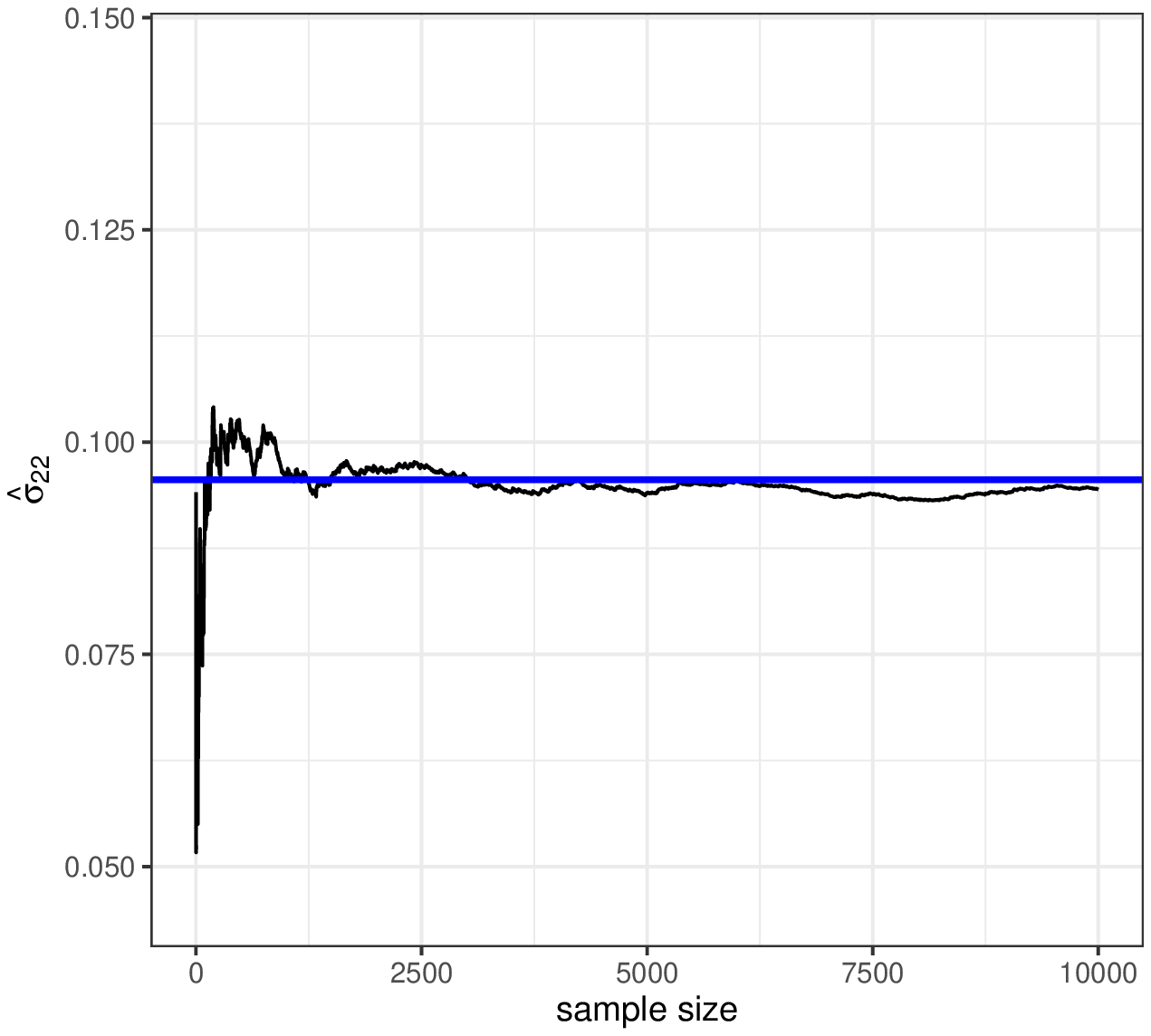}
\caption{Simulation study. Contour plot for the TSUT density (upper left corner) and trace plots of the evolution of the MC estimates for the mean and variance-covariance elements of $\Y$. The solid line represent the true estimated value by our proposal.}
\label{cap4:figsims}
\end{figure*}

In order to illustrate our method, we performed a simple Monte Carlo (MC) simulation study to show how
MC estimators for the mean vector and variance-covariance matrix elements converge to the real values computed by our method.

We consider a bivariate TSUT distribution $\Y \sim TSUT_{2,2}(\bmu,\bSigma,\bLambda,\btau,\nu,\bPsi;(\ap,\bp))$ with lower and upper truncation limits $\ap = (-0.8,-0.6)^\top$ and $\bp = (0.5,0.7)^\top$ respectively, null location vector $\bmu = \zero$, degrees of freedom $\nu=4$,
\begin{equation*}\label{cap4:pars}
\btau = \left(\begin{array}{cc}
-1\\
2
\end{array}
\right)
\text{,}
\quad
\bSigma = \left(\begin{array}{cc}
1 & 0.2 \\
0.2 & 4
\end{array}
\right),
\quad
\bLambda = \left(\begin{array}{cc}
1 & 3 \\
-3 & -2
\end{array}
\right)
\quad
\text{and}
\quad
\bPsi = \left(\begin{array}{cc}
1 & -0.5 \\
-0.5 & 1
\end{array}
\right).
\end{equation*}

Figure \ref{cap4:figsims} shows the contour plot for the TSUT density (upper left corner) as well as the evolution trace of the MC estimates for the mean (first row) and variance-covariance (last row) elements $\mu_1$, $\mu_2$, $\sigma_{11}$, $\sigma_{12}$ and $\sigma_{22}$. Estimated true values for the mean vector and the variance-covariance matrix  were computed using equations \eqref{cap4:eq:MEANVAR_EST1} and \eqref{cap4:eq:MEANVAR_EST2}, being
$$
\EE[\Y] = \left(\begin{array}{cc}
-0.039\\
0.303
\end{array}
\right)
\qquad\text{and}\qquad
\text{cov}[\Y] = \left(\begin{array}{cc}
0.112 & -0.007 \\
-0.007 & 0.096
\end{array}
\right),
$$
which are depicted as a blue solid line in Figure \ref{cap4:figsims}. Note that even with 1000 MC simulations there exists a significant variation in the chains.

\section{Additional results related to interval censored mechanism}\label{cap4:lemmas}

Under interval censoring mechanism the implementation of inferences depends on the computation of certain marginal and conditional expectations (\cite{Matos.SINICA}). For instance, for $\X = (\X_1^\top,\X_2^\top)^\top \sim \phi_{1+p}(\bxi,\bOmega,\nu)$, as in \eqref{cap4:xiomega}, with $\bPsi = 1$, $\bLambda = \blambda$ and $\btau=0$, it holds that $f_{\X_1}(\zero \mid \X_2=\Y) = \phi
\big(
\blambda^{\top}\bSigma^{-1/2}(\mathbf{Y}-\bmu)
\big)$
. Then,
\begin{equation}\label{xx1}
\mathbb{E}
\left[
g(\Y)
\frac{
f_{\X_1}(\zero \mid \X_2=\Y)
}{
\PP(\X_1>\zero \mid \X_2=\Y)
}
\right] = \mathbb{E}
\left[
g(\Y)
\frac{
\phi
\big(
\blambda^{\top}\bSigma^{-1/2}(\mathbf{Y}-\bmu)
\big)
}{
\Phi
\big(
\blambda^{\top}\bSigma^{-1/2}(\mathbf{Y}-\bmu)
\big)
}
\right],
\end{equation}
where $g(\cdot)$ is a measurable function. The expectation in the right side of the expression (\ref{xx1}) is highly used to perform inferences under SN censored models from a likelihood-based perspective, such as the E-Step of the EM-algorithm (\cite{Dempster77}).

Next, we derive general expressions that are involved in interval censored modeling, specifically, in the E-step of the EM algorithm. These expressions arise, when we consider the responses $\Y_i,\,i=1,\ldots,n$, to be i.i.d. realizations from a selection elliptical distribution or any of its particular cases. For instance, a SUT, EST or ST distribution or any normal limiting case as the SUN, ESN or SN distribution as the example in (\ref{xx1}).\\

\begin{lemma}\label{cap4:lema1}
Let
$\X=(\X_1^\top,\X_2^\top)^\top \sim EC_{q+p}(\bxi,\bOmega,h^{(q+p)})$ and
$\Y \sim TSLCT\text{-}EC_{p,q}(\bxi,\bOmega,h^{(q+p)},$
$C;(\ap,\bp))$ with truncation subset $C= C(\zero)$. For any measurable function $g(\y): \RR^p \rightarrow \RR$, we have that
\begin{equation}\label{cap4:eq:lema1}
\mathbb{E}
\left[
g(\Y)
\frac{
f_{\X_1}(\zero \mid \X_2=\Y)
}{
\PP(\X_1>\zero \mid \X_2=\Y)
}
\right] = \frac{\PP(\ap\leq\W_0\leq\bp)}{\PP(\ap\leq\Y_0\leq\bp)}
\frac{\mathbb{E}[g(\W)]}{\PP(\X_1 \geq \zero)} f_{\X_1}(\zero),
\end{equation}
where $\X_1 \sim EC_p(\bxi_{1},\bOmega_{11},h^{(q)})$, $\Y_0 \sim SLCT\text{-}EC_{p,q}(\bxi,\bOmega,h^{(q+p)},C(\zero))$,
$\W_0
\sim EC_p(\bxi_{2} - \bOmega_{21} \bOmega_{11}^{-1}\bxi_{1},\bOmega_{22}-\bOmega_{21}\bOmega_{11}^{-1}\bOmega_{21},h^{(p)}_{\zero})$
and
$\W \eqdist \W_0 \mid (\ap \leq \W_0 \leq \bp)$.
\end{lemma}

\medskip

\begin{proof}
Using basic probability theory, we have
\begin{align*}
&=
\mathbb{E}
\left[
g(\Y)
\frac{
f_{\X_1}(\zero \mid \X_2=\Y)
}{
\PP(\X_1>\zero \mid \X_2=\Y)
}
\right]\\
&=
\frac{1}{\PP(\ap\leq\Y_0\leq\bp)}
\int_{\ap}^{\bp}
g(\y)
\frac{
f_{\X_1}(\zero \mid \X_2=\y)
}{
\PP(\X_1>\zero \mid \X_2=\y)
}
f_{\Y}(\y)
\dr \y,
\\
&=
\frac{1}{\PP(\ap\leq\Y_0\leq\bp)}
\int_{\ap}^{\bp}
g(\y)
\frac{
f_{\X_1}(\zero \mid \X_2=\y)
}{
\PP(\X_1>\zero \mid \X_2=\y)
}
\frac{
\PP(\X_1 > \zero \mid \X_2=\y)f_{\X_2}(\y)
}{
\PP(\X_1 > \zero)
}
\dr \y,\\
&=
\frac{1}{\PP(\ap\leq\Y_0\leq\bp)}
\int_{\ap}^{\bp}
g(\y)
\frac{
f_{\X_1}(\zero \mid \X_2=\y)
f_{\X_2}(\y)
}{
\PP(\X_1 > \zero)
}
\dr \y,\\
&=
\frac{1}{\PP(\ap\leq\Y_0\leq\bp)}
\frac{f_{\X_1}(\zero)}{\PP(\X_1 > \zero)}
\int_{\ap}^{\bp}
g(\y)
f_{\X_2}(\y\mid \X_1 = \zero )\,
\dr \y,\\
&=
\frac{\PP(\ap\leq\W_0\leq\bp)}{\PP(\ap\leq\Y_0\leq\bp)}
\frac{\mathbb{E}[g(\W)]}{\PP(\X_1 > \zero)}f_{\X_1}(\zero),
\end{align*}
where $\W_0 \eqdist \X_2|(\X_1=\zero$) and
$\W \eqdist \W_0 \mid (\ap \leq \W_0 \leq \bp)$.
\end{proof}

\bigskip

\begin{lemma}\label{cap4:lema2}
Consider $\X$, $\Y$ and $g$ as in Lemma \ref{cap4:lema1}. Now, consider $\Y$
to be partitioned as $\Y=(\Y^{\top}_1,\Y^{\top}_2)^{\top}$ of
dimensions $p_1$ and $p_2$ ($p_1+p_2=p$). {For a given random variable $\U$, let $\U^*$ stands for $\U^*\eqdist \U\mid \Y_1$}. It follows that
\begin{equation}\label{cap4:eq:lema2}
\mathbb{E}
\left.
\left[
g(\Y_2)
\frac{
f_{\X_1}(\zero \mid \X_2=\Y)
}{
\PP(\X_1>\zero \mid \X_2=\Y)
}
\right| \Y_1
\right] = \frac{\PP(\ap_2\leq\W_{0}^{*}\leq\bp_2)}{\PP(\ap_2\leq\Y_{0}^{*}\leq\bp_2)}
\frac{\mathbb{E}[g(\W_2)]}{\PP(\X_{1}^{*} > \zero )}f_{\X_1^*}(\zero)
\end{equation}
with $\X_1$, $\Y_0$, and $\W_0$ as defined in Lemma \ref{cap4:lema1}, and
$\W_2 \eqdist \W_0^* \mid (\ap_2 \leq \W_0^* \leq \bp_2)$.
\end{lemma}

\begin{proof}
Consider $\X_2$ partitioned as $\X_2=(\X_{21}^\top,\X_{22}^\top)^\top$ such that $dim(\X_{21})=dim(\Y_1)$ and $dim(\X_{22})=dim(\Y_2)$. Since $f_{\Y_2}(\y_2|\Y_1=\y_1) = f_{\Y}(\y)/f_{\Y_{1}}(\y_1)$, it follows (in a similar manner that the proof of Lemma \ref{cap4:lema1}) that
\begin{align*}
&=\mathbb{E}
\left.
\left[
g(\Y_2)
\frac{
f_{\X_1}(\zero \mid \X_2=\Y)
}{
\PP(\X_1>\zero \mid \X_2=\Y)
}
\right| \Y_1
\right]
\\
&=
\frac{1}{\PP(\ap_2\leq\Y_0^*\leq\bp_2)}
\int_{\ap_2}^{\bp_2}
g(\y_2)
\frac{
f_{\X_1}(\zero \mid \X_2=\y)
}{
\PP(\X_1>\zero \mid \X_2=\y)
}
\frac{
\PP(\X_1 > \zero \mid \X_2=\y)
}{
\PP(\X_1 > \zero \mid \X_{21}=\y_1)
}
\frac{
f_{\X_2}(\y)
}{
f_{\X_{21}}(\y_1)
}
\dr \y_2,\\
&=
\frac{1}{\PP(\ap_2\leq\Y_{0}^{*}\leq\bp_2)}
\int_{\ap_2}^{\bp_2}
g(\y_2)
\frac{
f_{\X_1}(\zero \mid \X_2=\y)
}{
{\PP(\X_1 > \zero \mid \X_{21}=\y_1)}
}
\frac{
f_{\X_2}(\y)
}{f_{\X_{21}}(\y_1)
}
\dr \y_2,\\
&=
\frac{1}{\PP(\ap_2\leq\Y_{0}^{*}\leq\bp_2)}
\frac{f_{\X_1}(\zero)}{\PP(\X_1 > \zero \mid \X_{21}=\y_1)}
\int_{\ap_2}^{\bp_2}
g(\y_2)
\frac{f_{\X_2}(\y\mid \X_1 = \zero )}
{
f_{\X_{21}}(\y_1)
}
\,
\dr \y_2,\\
&=
\frac{1}{\PP(\ap_2\leq\Y_{0}^{*}\leq\bp_2)}
\frac{f_{\X_{1}}(\zero| \X_{21}=\y_1)}{\PP(\X_1 > \zero \mid \X_{21}=\y_1)}
\int_{\ap_2}^{\bp_2}
g(\y_2)
f_{\X_{22}}(\y_2\mid \X_{21} = \y_1,\X_1 = \zero )
\,
\dr \y_2,\\
&=
\frac{\PP(\ap_2\leq\W_{0}^{*}\leq\bp_2)}{\PP(\ap_2\leq\Y_{0}^{*}\leq\bp_2)}
\frac{\mathbb{E}[g(\W_2)]}{\PP(\X_{1}^{*} > \zero )}f_{\X_1^*}(\zero),
\end{align*}
where $\W_0^* \eqdist \X_{22}|(\X_{21}=\y_1,\X_1=\zero$) and
$\W_2 \eqdist \W_0^* \mid (\ap_2 \leq \W_0^* \leq \bp_2)$.
\end{proof}

In the next corollaries, we particularize the aforementioned lemmas to the truncated SUT, EST, SUN and ESN distributions.

\begin{corollary}Under the condition of Lemma  \ref{cap4:lema1}, \label{cap4:corlema1}
let $\bY \sim TSUT_{p,q}(\bmu,\bSigma,\bLambda,\btau,\nu,\bPsi,(\ap,\bp))$. For any measurable function $g(\y): \RR^p \rightarrow \RR$, we have that
\begin{equation}\label{cap4:eq:corlema1}
\mathbb{E}
\left[
g(\Y)
\frac{
t_q
\big(
(
\btau +
\bLambda^{\top}\bSigma^{-1/2}(\mathbf{Y}-\bmu)
)
\,
\nu(\Y)
,\bPsi
;\nu + p
\big)
}{
T_q
\big(
(
\btau +
\bLambda^{\top}\bSigma^{-1/2}(\mathbf{Y}-\bmu)
)
\,
\nu(\Y)
,\bPsi
;\nu + p
\big)
}
\right] = \frac{\PP(\ap\leq\W_0\leq\bp)}{\PP(\ap\leq\Y_0\leq\bp)}
\mathbb{E}[g(\W)]\neta,
\end{equation}
where $\neta = {t_q(\btau;\bPsi + \bLambda^\top\bLambda,\nu)}/
{T_q(\btau;\bPsi + \bLambda^\top\bLambda,\nu)}$,
$\Y_0 \sim SUT_{p,q}(\bmu,\bSigma,\bLambda,\btau,\nu,\bPsi)$, $\W_0 \sim t_p(\bgamma,\omega_{\tau}\bGamma,\nu+q)$ and
$\W \eqdist \W_0 \mid (\ap \leq \W_0 \leq \bp)$. When $\btau=\zero$, we have that $\neta = 2\,{t_q(\btau;\bPsi + \bLambda^\top\bLambda,\nu)}$ and
$\W_0 \sim t_p(\bmu,\nu\bGamma/(\nu+q),\nu+q)$
.\\

In particular for $q=1$, $\bY \sim \TEST_p(\bmu,\bSigma,\blambda,\tau,\nu;(\ap,\bp))$, and
\begin{equation}\label{cap4:eq:corlema1q1}
\mathbb{E}
\left[
g(\Y)
\frac{
t_1
\big(
(
\tau +
\blambda^{\top}\bSigma^{-1/2}(\mathbf{Y}-\bmu)
)
\,
\nu(\Y)
;\nu + p
\big)
}{
T_1
\big(
(
\tau +
\blambda^{\top}\bSigma^{-1/2}(\mathbf{Y}-\bmu)
)
\,
\nu(\Y)
;\nu + p
\big)
}
\right] = \frac{\PP(\ap\leq\W_0\leq\bp)}{\PP(\ap\leq\Y_0\leq\bp)}
\eta\,\mathbb{E}[g(\W)],
\end{equation}
with $\eta = {t_1(\tau;1+\blambda^\top\blambda,\nu)}/
{T_1(\tautil;\nu)}$,
$\Y_0 \sim EST_{p}(\bmu,\bSigma,\blambda,\btau,\nu)$, $\W_0 \sim t_p(\bgamma,(\nu + \tilde\tau^2)\bGamma/(\nu+1),\nu+1)$, and
$\W \eqdist \W_0 \mid (\ap \leq \W_0 \leq \bp)$. Similarly, when $\tau=0$, we have that $\eta = 2\,{t_1(0;1 + \blambda^\top\blambda,\nu)}$
and $\W_0 \sim t_p(\bmu,\nu\bGamma/(\nu+1),\nu+1)$.
\end{corollary}

\medskip

\begin{corollary} Under the condition of Lemma  \ref{cap4:lema1},
\label{cap4:corlema1}
let $\nu\rightarrow\infty$, $\bY \sim TSUN_{p,q}(\bmu,\bSigma,\bLambda,\btau,\bPsi,(\ap,\bp))$, it follows that
\begin{equation}\label{cap4:eq:corlema2}
\mathbb{E}
\left[
g(\Y)
\frac{
\phi_q
\big(
\btau +
\bLambda^{\top}\bSigma^{-1/2}(\mathbf{Y}-\bmu)
,\bPsi
\big)
}{
\Phi_q
\big(
\btau +
\bLambda^{\top}\bSigma^{-1/2}(\mathbf{Y}-\bmu)
,\bPsi
\big)
}
\right] = \frac{\PP(\ap\leq\W_0\leq\bp)}{\PP(\ap\leq\Y_0\leq\bp)}
\mathbb{E}[g(\W)]\neta,
\end{equation}
where $\neta = {\phi_q(\btau;\bPsi + \bLambda^\top\bLambda)}/
{\Phi_q(\btau;\bPsi + \bLambda^\top\bLambda)}$,
$\Y_0 \sim SUN_{p,q}(\bmu,\bSigma,\bLambda,\btau,\bPsi)$, $\W_0 \sim N_p(\bgamma,\bGamma)$, and
$\W \eqdist \W_0 \mid (\ap \leq \W_0 \leq \bp)$ .When $\btau=\zero$, we have that $\neta = 2\,{\phi_q(\zero;\bPsi + \bLambda^\top\bLambda)}$ and
$\W_0 \sim N_p(\bmu,\bGamma)$.\\

In particular for $q=1$, $\bY \sim \TESN_p(\bmu,\bSigma,\blambda;(\ap,\bp))$, and
\begin{equation}\label{cap4:eq:corlema2q1}
\mathbb{E}
\left[
g(\Y)
\frac{
\phi
\big(
\tau +
\blambda^{\top}\bSigma^{-1/2}(\mathbf{Y}-\bmu)
\big)
}{
\Phi
\big(
\tau +
\blambda^{\top}\bSigma^{-1/2}(\mathbf{Y}-\bmu)
\big)
}
\right] = \frac{\PP(\ap\leq\W_0\leq\bp)}{\PP(\ap\leq\Y_0\leq\bp)}
\eta\,\mathbb{E}[g(\W)],
\end{equation}
with $\eta = {\phi(\tau;1+\blambda^\top\blambda)}/
{\Phi(\tautil)}$,
$\Y_0 \sim ESN_{p}(\bmu,\bSigma,\blambda,\btau)$, $\W_0 \sim N_p(\bgamma,\bGamma)$, and
$\W \eqdist \W_0 \mid (\ap \leq \W_0 \leq \bp)$. Similarly, when $\tau=0$, we have that $\eta = \sqrt{2/\pi(1+\blambda^{\scriptscriptstyle\top}\blambda)}$
and $\W_0 \sim N_p(\bmu,\bGamma)$.
\end{corollary}

\section{Application of SE truncated moments on tail conditional expectation}


Let $Y$ be a random variable representing in this context, the total loss in a portfolio investment, a credit score, etc. Let $y_\alpha$ be the $(1-\alpha)$th quantile of $Y$, that is, $\PP(Y > y_\alpha) = \alpha$. Hence, the tail conditional expectation (TCE) (see, e.g., \cite{denuit2006actuarial}) is denoted by
\begin{equation}\label{TCE}
TCE_Y(y_\alpha) = \EE[Y\mid Y > y_\alpha].
\end{equation}

This can be interpreted as the expected value of the $\alpha$\% worse losses. The quantile $y_\alpha$ is usually chosen to be high in order to be pessimistic, for instance, $\alpha = 0.05$. Notice that, if we consider a variable $Y$ which we are interested on maximizing, for example, the pay-off of a portfolio, we simply compute $TCE_{-Y}(-y_\alpha) = -\EE[Y\mid Y \leq - y_\alpha]$, being a measure of worst expected income.

Main applications of TCE are in actuarial science and financial economics: market risk, credit risk of a portfolio, insurance, capital requirements for financial institutions, among others. TCE (also known as tail value at risk, TVaR) and it represents an alternative to the traditional value at risk (VaR) that is more sensitive to the shape of the tail of the loss distribution. Furthermore, if $Y$ is a continuous r.v., TCE coincides with the well-known risk measure expected shortfall (\cite{acerbi2002expected}).
In contrast with VaR, TCE is said to be a coherent measure, holding desirable mathematical properties in the context of risk measurement and and is a convex function of the selection weights (\cite{artzner1999coherent,pflug2000some}).  A good reference to several risk measures and their properties can be found in \cite{sereda2010distortion}.

\paragraph*{\bf Multivariate framework}

Let consider a set of $p$ assets, business lines, credit scores, $\Y = (Y_1,\cdots,Y_p)^\top$. In the multivariate case, the sum of risks arises as a natural and simple measure of total risk. Hence, the sum $S = Y_1+Y_2+\cdots+Y_p$ follows a univariate distribution and from \eqref{TCE}, we have that the TCE for $S$ is given by
\begin{equation}\label{cap4:TCE_S}
TCE_S(s_\alpha) = 
\EE[S\mid S> s_\alpha].
\end{equation}

\noindent Even though we may know the marginal distribution of $S$, it is preferable to compute the total risk $TCE$ of $S$ as a decomposed sum, that is
\begin{equation}\label{cap4:TCE_SYi}
\EE[S\mid S> s_\alpha] = \sumip
\EE[Y_i \mid S> s_\alpha],
\end{equation}
where each term $\EE[Y_i \mid S> s_\alpha]$ represents the average amount of risk due to $Y_i$. This decomposed sum offers a way to study the individual impact of the elements of the set, being an improvement to \eqref{cap4:TCE_S}.

In order to model combinations of correlated risks, \cite{landsman2003tail} extended the TCE to the multivariate framework. The multivariate TCE (MTCE) is given by 
\begin{equation}\label{MTCE}
MTCE_\Y(\y_{\balpha}) = \EE[\Y \mid \Y > \y_{\balpha}] = \EE[\Y \mid Y_1 > {y}_{1\alpha_1},\ldots,Y_p > {y}_{p\alpha_p}],
\end{equation}
with $\balpha = (\alpha_1,\ldots,\alpha_p)$ be a vector of quantiles of interest. Notice that the quantile-level for the MTCE is fixed per each risk $\ii$, in contrast with the TCE of the sum, which is fixed over all the sum of risk $S$. 

\subsection{MTCE for selection elliptical distributions}

Let consider  $\Y \sim SLCT\text{-}EC_{p,q}(\bxi,\bOmega,h^{(q+p)},C)$. With loss of generality, we consider the selection subset $C = C(\zero)$. It follows from Theorem \ref{cap4:theo.1} that
\begin{equation}\label{MTCE_selection}
MTCE_\Y(\y_{\balpha}) = \EE[\X_2 \mid \X > \x_{\balpha}],
\end{equation}
with $\x_{\balpha} = (\zero_q^\top,\y_{\balpha}^\top)^\top$ and where $\X  = (\X_1^{\top},\X_2^{\top})^\top \sim EC_{q+p}(\bxi,\bOmega,h^{(q+p)})$. It is noteworthy that the computation of the MTCE for $\Y$ following a SE distribution relies on the calculation of truncated moments for its symmetrical elliptical multivariate case. 

On the other hand, by noticing that $S = \one^\top\Y$, it follows from \eqref{cap4:AY+b} that $S$ is an
univariate SE distribution given by $
S \sim SLCT\text{-}EC_{1,q}(\bxi_s,\bOmega_s,h^{(q+1)},C)
$, with
\begin{equation*}
\bxi_S = \left(\begin{array}{cc}
\bxi_{1}\\
\one^\top\bxi_{2}
\end{array}
\right)
\qquad
\text{and}
\qquad
\bOmega_S = \left(\begin{array}{cc}
\bOmega_{11} & \bOmega_{12}\one \\
\one^\top\bOmega_{21} & \one^\top\bOmega_{22}\one
\end{array}
\right).
\end{equation*}

Hence, its TCE in \eqref{cap4:TCE_S} can be easily computed as 
$\EE[S\mid S> s_\alpha] = \EE[W_2 \mid \W_1 > \zero, W_2 > s_\alpha]
$, $\W  = (\W_1^\top,W_2)^\top \sim EC_{q+1}(\bxi_s,\bOmega_s,h^{(1+q)})$, due to $S \eqdist W_2 \mid (\W_1 > \zero)$. Next, we establish a general proposition for computing $\EE[S\mid S> \alpha_s]$ in matrix form as a decomposed sum.\\

\begin{proposition}\label{cap4:prop:main}
Let $\Y \sim SLCT\text{-}EC_{p,q}(\bxi,\bOmega,h^{(q+p)},C)$, with $\bxi$ and $\bOmega$ as in \eqref{cap4:sel.ec}, and 
$\W  = (\W_1^\top,W_2)^\top \sim EC_{q+1}(\bxi_S,\bOmega_S,h^{(1+q)})$ as before.
It follows that
\begin{align}\label{cap4:app:prop1}
\EE[S\mid S> s_\alpha] 
&= \one^\top\s
,
\end{align}
with $\s = \bxi_2 + \bOmega_{2S}\,\bOmega_S^{-1}(\boldsymbol{\mathcal{E}}_S-\bxi_S)$, where $\bOmega_{2S} = (\bOmega_{21},\bOmega_{22}\one)$ and $\boldsymbol{\mathcal{E}}_S = \EE[\W\mid \W_1 > \zero, W_2 > s_\alpha]$.
\end{proposition}

\begin{proof}
Let $\A = (\one,\bI_p)^\top$ be a real matrix of dimensions $(p+1)\times p$. For $\V = \A\Y$, it follows that

\begin{equation}\label{cap4:sel.ec}
\V =  \left(\begin{array}{cc}
V_1 \\
\V_2
\end{array}
\right)
\sim
SLCT\text{-}EC_{p+1,q}
\left(
\bxi_V = \left(\begin{array}{cc}
\bxi_{S}\\
\bxi_{2}\\
\end{array}
\right),
\bOmega_V = \left(\begin{array}{cc}
\bOmega_{S} & \bOmega_{2S}^\top \\
\bOmega_{2S} & \bOmega_{22}
\end{array}
\right),
h^{(q+1+p)},
C
\right),
\end{equation}
where $\V = (S,\Y^\top)^\top$. It comes from the definition of selection distribution that $\V \eqdist (X_2,\X_3^\top)^\top|(\X_1 > 0)$, where $\X = (\X_1^\top,X_2,\X_3^\top)^\top$ is a partitioned random vector with elements of dimensions $q$, $1$ and $p$ respectively, where $\X\sim EC_{p+q+1}
(\bxi_V,\bOmega_V;h^{(q+1+p)})$. Hence, it is straightforward to see that 
\begin{align*}
\s &= \EE[\Y\mid S> s_\alpha] = \EE[\X_3\mid \X_1 > \zero, X_2 > s_\alpha, -\binfty \leq \X_3 \leq \binfty].
\end{align*}
Since there exists a non-truncated partition, the result in \eqref{cap4:app:prop1} then immediately follows from equation \eqref{cap4:cond_inf_mean}, with $\W = (\X_1,X_2)^\top$.

\end{proof}

\begin{remark}
It is noteworthy that, the $i$th element of vector $\s$, say $s_i = \ep_i^\top \s$, is equal to $\EE[Y_i \mid S> \alpha_s]$, representing the contribution to the total risk due to the $i$th risk.
\end{remark}

\begin{remark}
Since $S \eqdist W_2 \mid (\W_1 > \zero)$, it follows that the last element of the vector $\boldsymbol{\mathcal{E}}_s$ is equivalent to $\EE[S\mid S> s_\alpha] = \EE[W_2 \mid \W_1 > \zero, W_2 > s_\alpha].
$
\end{remark}

\subsection{\bf Application of MTCE using a ST distribution}

Suppose that a set of risks $\Y$ are distributed as $\Y \sim ST_p(\bmu,\bSigma,\blambda,\nu)$. Let $\y$ represents a realization of $\Y$. Based on $\y$, the set of parameters $\btheta = (\bmu,\bSigma,\blambda,\nu)^\top$ can be estimated through maximum likelihood estimation.
It follows that
\begin{equation}\label{MTCE_STcase}
MTCE_\Y(\y_{\balpha}) = \EE[\X_2 \mid X_1 > 0, \X_2 > \y_{\balpha}],
\end{equation}
where $\X  = (X_1,\X_2^{\top})^\top \sim t_{1+p}(\bxi,\bOmega,\nu)$ with 
\begin{equation}\label{cap4:xiomegaST}
\bxi = \left(\begin{array}{cc}
0\\
\bmu
\end{array}
\right)
\qquad
\text{and}
\qquad
\bOmega = \left(\begin{array}{cc}
1 & \bDelta^\top \\
\bDelta & \bSigma
\end{array}
\right).
\end{equation}

Additionally, using simple algebraic manipulation, it follows from
\eqref{cap4:AY+b}
that
\begin{equation}\label{cap4:STapp}
S
\sim
ST_1
\left(
\mu_S = \sumip \mu_i,
\,
\sigma^2_S = \sumip \sum_{j=1}^p \sigma_{ij},
\,
\lambda_S = \frac{\Delta_S}{\sqrt{\sigma^2_S - \Delta_S^2}}
,
\,\nu
\right),
\end{equation}
with $\Delta_S = \sumip \Delta_i$. Besides, the TCE of the sum is given by 
$TCE_S(s_\alpha) = \EE[W_2 \mid W_1 > 0, W_2 > s_\alpha]
$, $\W  = (W_1^\top,W_2)^\top \sim t_{2}(\bxi_S,\bOmega_S,\nu)$, where
\begin{equation*}
\bxi_S = \left(\begin{array}{cc}
0\\
\mu_S
\end{array}
\right),
\qquad
\text{and}
\qquad
\bOmega_S = \left(\begin{array}{cc}
1 & \Delta_S \\
\Delta_S & \sigma_S^2
\end{array}
\right).
\end{equation*}
Finally, we have from Proposition \ref{cap4:prop:main} that 
\begin{align}\label{cap4:finalSi}
\EE[Y_i \mid S> \alpha_s],
&= 
\ep_i^\top\left[
\bmu + (\bDelta,\bSigma\one)\,\bOmega_S^{-1}(\boldsymbol{\mathcal{E}}_S-\bxi_S)
\right]
,
\nonumber
\\
&=
\mu_i +
\mathcal{E}_{S1}(\Delta_i\sigma^2_S + \sigma_{iS}\Delta_S)
- (TCE_S(s_\alpha) -\mu_S)(\Delta_i\Delta_S + \sigma_{iS})
,
\end{align}
with  $\mathcal{E}_{S1} = \EE[W_1\mid W_1 > 0, W_2 > s_\alpha]$ and 
$\sigma_{iS} = \sum_{j=1}^p \sigma_{ij}$. Besides,
\begin{align}\label{cap4:finalS}
\EE[S\mid S> s_\alpha] 
&= 
\mu_S
+
\mathcal{E}_{S1}
\sumip 
\left\{
\Delta_i\sigma^2_S + \sigma_{iS}\Delta_S
\right\}
-
(TCE_S(s_\alpha) -\mu_S)
\sumip 
\left\{
\Delta_i\Delta_S + \sigma_{iS}
\right\}.
\end{align}

\section{Conclusions}

In this paper, we proposed expressions to compute product moment of truncated multivariate distributions belonging to the selection elliptical family, showing in a clever way that their moments can be computed using an unique moment for their respective elliptical symmetric case.  In contrast with other recent works, we avoid cumbersome expressions, having neat formulas for high-order truncated moments. To the best of our knowledge, this is the first proposal discussing the conditions of existence of the truncated moments for members of the selection elliptical family. Also, we propose optimized methods able to deal with extreme setting of the parameters, partitions with almost zero volume or no truncation.
	
We expect in the near future to use expressions in Section \ref{cap4:lemmas} to propose a robust likelihood-based censored regression model considering EST errors, able to fit multivariate censored responses with high skewness/kurtosis, presence of atypical observations and missing data. As more truncated moments for other symmetric elliptical distributions appear in the literature, we expect to implement the truncated moments for their respective asymmetric extended versions as well as censored models considering this last. Additionally, theoretical results can be extended to compute the moments of the class of extended generalized skew-elliptical distributions (see, \cite{landsman2017extended}), where the jointly distributed condition in \eqref{cap4:sel.ec} is not longer considered.

Finally, theoretical and MC moments (among other functions of interest) for several multivariate asymmetric distributions are already available in our {\tt MomTrunc R} package, which will be constantly updated when other treatable distributions are available.

\section*{Acknowledgment} 
Christian Galarza acknowledges support from FAPESP-Brazil (Grant 2015/17110-9 and Grant 2018/11580-1).


\begin{thebibliography}{10}

\bibitem{tobin1958estimation}
James Tobin.
\newblock Estimation of relationships for limited dependent variables.
\newblock {\em Econometrica: Journal of the Econometric Society}, pages 24--36,
1958.

\bibitem{Tallis1961}
G.~M. Tallis.
\newblock The moment generating function of the truncated multi-normal
distribution.
\newblock {\em Journal of the Royal Statistical Society. Series B (Statistical
Methodology)}, 23(1):223--229, 1961.

\bibitem{bg2009moments}
B.~G. Manjunath and Stefan Wilhelm.
\newblock Moments calculation for the double truncated multivariate normal
density.
\newblock {\em Available at SSRN 1472153}, 2009.

\bibitem{nadarajah2007truncated}
Saralees Nadarajah.
\newblock A truncated bivariate t distribution.
\newblock {\em Economic Quality Control}, 22(2):303--313, 2007.

\bibitem{lin2011some}
H.~J. Ho, T.~I. Lin, H.~Y. Chen, and W.~L. Wang.
\newblock Some results on the truncated multivariate t distribution.
\newblock {\em Journal of Statistical Planning and Inference}, 142:25--40,
2012.

\bibitem{ARISMENDI201729}
Juan~C. Arismendi and Simon Broda.
\newblock Multivariate elliptical truncated moments.
\newblock {\em Journal of Multivariate Analysis}, 157:29 -- 44, 2017.

\bibitem{kan2017moments}
Raymond Kan and Cesare Robotti.
\newblock On moments of folded and truncated multivariate normal distributions.
\newblock {\em Journal of Computational and Graphical Statistics},
25(1):930--934, 2017.

\bibitem{roozegar2020moments}
Roohollah Roozegar, Narayanaswamy Balakrishnan, and Ahad Jamalizadeh.
\newblock On moments of doubly truncated multivariate normal mean--variance
mixture distributions with application to multivariate tail conditional
expectation.
\newblock {\em Journal of Multivariate Analysis}, 177:104586, 2020.

\bibitem{barndorff1982normal}
Ole Barndorff-Nielsen, John Kent, and Michael S{\o}rensen.
\newblock Normal variance-mean mixtures and z distributions.
\newblock {\em International Statistical Review/Revue Internationale de
Statistique}, pages 145--159, 1982.

\bibitem{breymann2013ghyp}
Wolfgang Breymann and David L{\"u}thi.
\newblock ghyp: A package on generalized hyperbolic distributions.
\newblock {\em Manual for R Package ghyp}, 2013.

\bibitem{arellano2010multivariate}
Reinaldo~B Arellano-Valle and Marc~G Genton.
\newblock Multivariate extended skew-t distributions and related families.
\newblock {\em Metron}, 68(3):201--234, 2010.

\bibitem{AzzaliniC2003}
A.~Azzalini and A.~Capitanio.
\newblock Distributions generated and perturbation of symmetry with emphasis on
the multivariate skew-t distribution.
\newblock {\em Journal of the Royal Statistical Society, Series B},
61:367--389, 2003.

\bibitem{arellano2006unified}
Reinaldo~B Arellano-Valle, M{\'a}aarcia~D Branco, and Marc~G Genton.
\newblock A unified view on skewed distributions arising from selections.
\newblock {\em Canadian Journal of Statistics}, 34(4):581--601, 2006.

\bibitem{arellano2006unification}
Reinaldo~B Arellano-Valle and Adelchi Azzalini.
\newblock On the unification of families of skew-normal distributions.
\newblock {\em Scandinavian Journal of Statistics}, 33(3):561--574, 2006.

\bibitem{BrancoD2001}
M.~D. Branco and D.~K. Dey.
\newblock A general class of multivariate skew-elliptical distributions.
\newblock {\em Journal of Multivariate Analysis}, 79:99--113, 2001.

\bibitem{Moran-vasquez2019}
Ra{\'{u}}l~Alejandro Mor{\'{a}}n-V{\'{a}}squez and Silvia L~P Ferrari.
\newblock {New results on truncated elliptical distributions}.
\newblock {\em Communications in Mathematics and Statistics}, (53), 2019.

\bibitem{Matos.SINICA}
L.~A. Matos, M.~O. Prates, M.~H. Chen, and V.~H. Lachos.
\newblock Likelihood-based inference for mixed-effects models with censored
response using the multivariate-t distribution.
\newblock {\em Statistica Sinica}, 23:1323--1342, 2013.

\bibitem{Dempster77}
A.~Dempster, N.~Laird, and D.~Rubin.
\newblock Maximum likelihood from incomplete data via the {EM} algorithm.
\newblock {\em Journal of the Royal Statistical Society, Series B}, 39:1--38,
1977.

\bibitem{denuit2006actuarial}
Michel Denuit, Jan Dhaene, Marc Goovaerts, and Rob Kaas.
\newblock {\em Actuarial theory for dependent risks: measures, orders and
models}.
\newblock John Wiley \& Sons, 2006.

\bibitem{acerbi2002expected}
Carlo Acerbi and Dirk Tasche.
\newblock Expected shortfall: a natural coherent alternative to value at risk.
\newblock {\em Economic Notes}, 31(2):379--388, 2002.

\bibitem{artzner1999coherent}
Philippe Artzner, Freddy Delbaen, Jean-Marc Eber, and David Heath.
\newblock Coherent measures of risk.
\newblock {\em Mathematical Finance}, 9(3):203--228, 1999.

\bibitem{pflug2000some}
Georg~Ch Pflug.
\newblock Some remarks on the value-at-risk and the conditional value-at-risk.
\newblock In {\em Probabilistic Constrained Optimization}, pages 272--281.
Springer, 2000.

\bibitem{sereda2010distortion}
Ekaterina~N Sereda, Efim~M Bronshtein, Svetozar~T Rachev, Frank~J Fabozzi, Wei
Sun, and Stoyan~V Stoyanov.
\newblock Distortion risk measures in portfolio optimization.
\newblock In {\em Handbook of portfolio construction}, pages 649--673.
Springer, 2010.

\bibitem{landsman2003tail}
Zinoviy Landsman and Emiliano~A Valdez.
\newblock Tail conditional expectations for elliptical distributions.
\newblock {\em North American Actuarial Journal}, 7(4):55--71, 2003.

\bibitem{landsman2017extended}
Zinoviy Landsman, Udi Makov, and Tomer Shushi.
\newblock Extended generalized skew-elliptical distributions and their moments.
\newblock {\em Sankhya A}, 79(1):76--100, 2017.

\end{thebibliography}

\end{document}